\documentclass[11pt, a4paper]{amsart}
\usepackage{amsthm, amssymb, mathtools, cite, graphicx, color,, subcaption, amsmath}
\usepackage{tikz}
\usepackage{tikz-cd}
\usepackage{rotating}
\usepackage{lipsum}

\usetikzlibrary{matrix}
\mathtoolsset{showonlyrefs}
\newtheorem{theorem}{Theorem}[section]
\newtheorem*{theorem*}{Theorem}

\newtheorem{proposition}[theorem]{Proposition}
\newtheorem{corollary}[theorem]{Corollary}
\newtheorem*{corollary*}{Corollary}
\newtheorem{definition}[theorem]{Definition}

\numberwithin{equation}{section}
\title{The pre-symplectic geometry of opers and the holonomy map}
\author{Andrew Sanders}

\begin{document}

\address{Department of Mathematics, Heidelberg University}
\email{asanders@mathi.uni-heidelberg.de}
\thanks{Sanders gratefully acknowledges partial support from U.S. National Science Foundation grants DMS 1107452, 1107263, 1107367 "RNMS: GEometric structures And Representation varieties" (the GEAR Network).}


\date{}

\begin{abstract}
In this paper, we construct the moduli space of marked oper structures on a closed, oriented smooth surface of negative Euler characteristic as a holomorphic fiber bundle over Teichm\"{u}ller space.  We prove that the holonomy map from the space of marked oper structures to the moduli space of reductive flat bundles is a holomorphic immersion, generalizing the known results for the moduli space of marked complex projective structures.  Finally, we prove that the symplectic structure on the moduli space of marked complex projective structures extends to a pre-symplectic structure on the moduli space of marked opers whose reduced phase space is the space of marked complex projective structures.
\end{abstract}
\maketitle

\section{Introduction}

A complex projective structure on a compact Riemann surface $X$ of negative Euler characteristic is a maximal atlas of holomorphic charts with values in $\mathbb{CP}^{1}$ whose transition functions are given by restrictions of M\"{o}bius transformations.  Varying the Riemann surface structure on the underlying smooth surface $\Sigma$, complex projective structures collect into the moduli space of marked complex projective structures $\mathcal{CP}_{\Sigma}$ which is a holomorphic affine bundle modelled on the cotangent bundle of Teichm\"{u}ller space $\mathcal{T}_{\Sigma}.$  

Furthermore, the moduli space admits a holonomy map 
\begin{align}
\mathcal{CP}_{\Sigma}\rightarrow
\textnormal{Hom}(\pi_{1}(\Sigma), \textnormal{PSL}_{2}(\mathbb{C}))// \textnormal{PSL}_{2}(\mathbb{C})
\end{align}
which is a local bi-holomorphism.  In this paper, we will generalize all of these results to the moduli space of $\mathrm{G}$-opers where $\mathrm{G}$ is a complex simple Lie group of adjoint type.

Given a reductive complex Lie group $\mathrm{G},$ Beilinson-Drinfeld \cite{BD91} introduced a higher rank generalization of complex projective structures, called $\mathrm{G}$-opers, which share many of their interesting properties.  For $\textnormal{SL}(n, \mathbb{C})$, these objects were previously studied in the arena of $n$-th order linear ordinary differential equations by Teleman \cite{TEL59}.

A $\mathrm{G}$-oper on a Riemann surface $X$ is a triple $(E_{\mathrm{G}}, E_{\mathrm{B}}, \omega)$, where $E_{\mathrm{G}}$ is a holomorphic principal $\mathrm{G}$-bundle on $X$, $E_{\mathrm{B}}$ is a holomorphic reduction to a Borel subgroup $\mathrm{B}<\mathrm{G}$, and $\omega$ is a holomorphic flat connection on $E_{\mathrm{G}}$ which satisfies a certain non-degeneracy condition with respect to the sub-bundle $E_{\mathrm{B}}$.  In the case that $\mathrm{G}=\textnormal{PSL}_{2}(\mathbb{C}),$ the notion of a $\mathrm{G}$-oper on $X$ reduces to the standard encoding of a complex projective structure on $X$ via a holormophic flat $\mathbb{CP}^{1}$-bundle over $X$ equipped with a holomorphic section transverse to the flat structure.

Fixing a connected, closed Riemann surface $X$ of genus at least two and a complex simple Lie group $\mathrm{G}$ of adjoint group, the space of $\mathrm{G}$-opers on $X$ has a (non-unique) parameterization by the \emph{Hitchin base} 
\begin{align}\label{base param}
\mathcal{B}_{X}(\mathrm{G}):=\bigoplus_{i=1}^{\ell} \textnormal{H}^{0}(X, \mathcal{K}^{m_{i}+1}).
\end{align}
Here, $\mathcal{K}$ is the canonical sheaf of holomorphic one forms on $X$ and the integers 
$1=m_{1}\leq m_{2}\leq... \leq m_{\ell}$ are the exponents of the Lie algebra $\mathfrak{g}$ of $\mathrm{G}.$  The situation for general semi-simple groups $\mathrm{G}$ is not very different, and amounts to taking products and discrete phenomena (see \cite{BD91}).  To avoid various Lie-theoretic subtleties, in this paper we focus on the case of simple groups of adjoint type.

By way of the parameterization \eqref{base param} by the Hitchin base, the space of $\mathrm{G}$-opers on $X$ acquires the the structure of a complex manifold, and as with complex projective structures, there is a \emph{holonomy} map to the space of gauge equivalence classes of $C^{\infty}$-flat $\mathrm{G}$-bundles on the smooth surface $\Sigma$ underlying the Riemann surface $X.$  

When $X$ is a closed, connected Riemann surface of genus at least two, Beilinson-Drinfeld proved \cite{BD91}, \cite{BD05} (see also \cite{WEN16}) that the holonomy map is a proper, holomorphic Lagrangian embedding for the Atiyah-Bott-Goldman \cite{AB83}\cite{GOL84} complex symplectic structure on the moduli space of flat reductive $\mathrm{G}$-bundles.

The primary goal of this paper is to extend the above results to the setting where the Riemann surface is allowed to vary, using the theory of complex projective structures as a guide.  

Let $\Sigma$ be a closed, oriented, smooth connected surface of genus at least two. A $\Sigma$-marked Riemann surface is a pair $X:=(\Sigma, J)$ comprising a complex structure $J$ on $\Sigma$ whose induced orientation agrees with the ambient orientation of $\Sigma.$ Two $\Sigma$-marked Riemann surfaces are isomorphic if they are bi-holomorphic via a diffeomorphism of $\Sigma$ isotopic to the identity. The Teichm\"{u}ller space $\mathcal{T}_{\Sigma}$ is the space parameterizing isomorphism classes of $\Sigma$-marked Riemann surfaces: $\mathcal{T}_{\Sigma}$ is a complex manifold of dimension $3g-3$ where $g$ is the genus of $\Sigma.$

 Abusing the equivalence relation, we sometimes refer to an element of $\mathcal{T}_{\Sigma}$ as a $\Sigma$-marked Riemann surface, and we call a $\mathrm{G}$-oper on a $\Sigma$-marked Riemann surface a $\Sigma$-marked $\mathrm{G}$-oper.  

Our first theorem constructs the moduli space of $\Sigma$-marked $\mathrm{G}$-opers as a complex manifold.

\begin{theorem}\label{def space}
Let $\mathrm{G}$ be a complex simple Lie group of adjoint type.  Then, there is a Hausd\"{o}rff complex manifold $\mathcal{O}\mathfrak{p}_{\Sigma}(\mathrm{G})$ parameterizing isomorphism classes of $\Sigma$-marked $\mathrm{G}$-opers and a holomorphic submersion
\begin{align}
\mathcal{O}\mathfrak{p}_{\Sigma}(\mathrm{G})\rightarrow \mathcal{T}_{\Sigma}.
\end{align}
The space $\mathcal{O}\mathfrak{p}_{\Sigma}(\mathrm{G})$ is a fine moduli space.
\end{theorem}

To prove this result, we take a complex-analytic approach by first constructing universal Kuranishi families \cite{KUR62} deforming a given $\Sigma$-marked $\mathrm{G}$-oper.  The bases of these universal families are used to simultaneously construct a topology and a coordinate atlas on $\mathcal{O}\mathfrak{p}_{\Sigma}(\mathrm{G}).$

Varying $X$ over the Teichm\"{u}ller space of $\Sigma$ gives rise to a holomorphic vector bundle $\mathcal{B}_{\Sigma}(\mathrm{G})$ over $\mathcal{T}_{\Sigma},$ whose fiber over $X$ is the associated Hitchin base $\mathcal{B}_{X}(\mathrm{G})$.  Our next theorem, which is an analogue of the parameterization result of Beilinson-Drinfeld \cite{BD91}, establishes the relation between the moduli space of $\mathrm{G}$-opers $\mathcal{O}\mathfrak{p}_{\Sigma}(\mathrm{G})$ and the bundle $\mathcal{B}_{\Sigma}(\mathrm{G}).$

\begin{theorem}\label{id h base}
Let $\mathrm{G}$ be a complex simple Lie group of adjoint type. There is a (natural in $\mathrm{G}$) commutative diagram
\begin{center}
\begin{tikzcd}
\mathcal{CP}_{\Sigma} \arrow{r} \arrow{d}
& \mathcal{O}\mathfrak{p}_{\Sigma}(\mathrm{G}) \arrow{dl} \\
\mathcal{T}_{\Sigma}.
\end{tikzcd}
\end{center}
Furthermore, every smooth section $s$ of the projection 
\begin{align}
\mathcal{CP}_{\Sigma} \rightarrow \mathcal{T}_{\Sigma}
\end{align}
induces a diffeomorphism
\begin{align}
\phi_{s}:\mathcal{O}\mathfrak{p}_{\Sigma}(\mathrm{G})\rightarrow \mathcal{B}_{\Sigma}(\mathrm{G})
\end{align}
commuting with the projections to $\mathcal{T}_{\Sigma}.$  If $s$ is holomorphic, then $\phi_{s}$ is a bi-holomorphism.
\end{theorem}

This theorem shows that the holomorphic fiber bundle $\mathcal{O}\mathfrak{p}_{\Sigma}(\mathrm{G})$ has a structure that resembles an affine bundle over $\mathcal{T}_{\Sigma}$ whose underlying vector bundle is $\mathcal{B}_{\Sigma}(\mathrm{G}).$  For $\mathrm{G}=\textnormal{PSL}_{2}(\mathbb{C})$, this is literally true, and it is well known that $\mathcal{O}\mathfrak{p}_{\Sigma}(\textnormal{PSL}_{2}(\mathbb{C}))\simeq \mathcal{CP}_{\Sigma}$ is a holomorphic affine bundle over $\mathcal{T}_{\Sigma}$ modeled on the cotangent bundle $T^{\star}\mathcal{T}_{\Sigma}\simeq \mathcal{B}_{\Sigma}(\textnormal{PSL}_{2}(\mathbb{C}))$ of the Teichm\"{u}ller space of $\Sigma.$

For general $\mathrm{G}$, instead of being able to subtract arbitrary elements in a fiber over $X\in \mathcal{T}_{\Sigma}$ of $\mathcal{O}\mathfrak{p}_{\Sigma}(\mathrm{G})\rightarrow \mathcal{T}_{\Sigma}$, one can only subtract an arbitrary element of the fiber from a $\textnormal{PSL}_{2}(\mathbb{C})$-oper on $X.$  This is intimately related to a fact we shall discuss later, namely that $\mathcal{O}\mathfrak{p}_{\Sigma}(\mathrm{G})$ admits a constant rank closed holomorphic $2$-form which is degenerate if $\textnormal{rk}(G)>1$, and for which the fibers of the projection $\mathcal{O}\mathfrak{p}_{\Sigma}(\mathrm{G})\rightarrow \mathcal{T}_{\Sigma}$ are maximal isotropic sub-manifolds.  

By the results in \cite{LS17}, if $(M,\omega)$ is a symplectic manifold and $N$ is a smooth manifold, a locally trivial fiber bundle $(M, \omega)\rightarrow N$ with Lagrangian fibers has the property that the fibers have a canonical flat affine structure.\footnote{This fact was known long before the paper \cite{LS17}, but the discussion in \cite{LS17} makes this issue very explicit.  Furthermore, the content of \cite{LS17} is closely related to the circumstances addressed in this paper.}  For pre-symplectic manifolds which are the total space of a maximally isotropic fibration, a weaker result is true.

 In light of this, the strange \emph{partially affine} structure on the fibers of $\mathcal{O}\mathfrak{p}_{\Sigma}(\mathrm{G})\rightarrow \mathcal{T}_{\Sigma}$ for general $\mathrm{G}$ is ultimately a reflection of the fact that $\mathcal{O}\mathfrak{p}_{\Sigma}(\mathrm{G})\rightarrow \mathcal{T}_{\Sigma}$ is a maximally isotropic fibration.  We hope that this discussion relieves some of the "suspiciousness" regarding the bijection between $\mathrm{G}$-opers on $X$ and $\mathcal{B}_{X}(\mathrm{G})$ 
referred to in \cite{BD05}[pg. 21].

Now we move on to a discussion of the forgetful map from the moduli space of $\Sigma$-marked $\mathrm{G}$-opers to the space of $C^{\infty}$-flat $\mathrm{G}$-bundles on $\Sigma$.  This is the map sending a $\Sigma$-marked $\mathrm{G}$-oper $(E_{\mathrm{G}}, E_{\mathrm{B}}, \omega, X)$ to the $C^{\infty}$-flat $\mathrm{G}$-bundle $(E_{\mathrm{G}}, \omega)$ over $\Sigma$.

Let $\mathcal{F}_{\Sigma}^{\star}(\mathrm{G})$ be the set of isomorphism classes of smooth, irreducible flat $\mathrm{G}$-bundles over $\Sigma$ with isotropy equal to the center of $\mathrm{G}$. Utilizing the results of Goldman \cite{GOL84} and standard techniques of defomormation theory, $\mathcal{F}_{\Sigma}^{\star}(\mathrm{G})$ has the structure of a complex symplectic manifold.  

The next theorem generalizes the (independent) classical result of Earle \cite{EAR81}, Hejhal \cite{HEJ78} and Hubbard \cite{HUB81} concerning the local injectivity of the holonomy map from the moduli space of marked complex projective structures on $\Sigma$ to the space of flat bundles $\mathcal{F}_{\Sigma}^{\star}(\textnormal{PSL}_{2}(\mathbb{C})).$
\begin{theorem}\label{immersion}
Let $\mathrm{G}$ be a complex simple Lie group of adjoint type.  Then, the holonomy map
\begin{align}
\textnormal{H}: \mathcal{O}\mathfrak{p}_{\Sigma}(\mathrm{G})\rightarrow \mathcal{F}_{\Sigma}^{\star}(\mathrm{G})
\end{align}
is a holomorphic immersion.
\end{theorem}
Our proof is similar in spirit to the proof of Hubbard \cite{HUB81}, with sheaf cohomology playing a central role.  In particular, we identify a complex of sheaves $\mathcal{A}^{\bullet}$ on $X$ whose hyper-cohomology governs the deformations of $(E_{\mathrm{G}}, E_{\mathrm{B}}, \omega, X)$.  We identify the derivative of the map
\begin{align}\label{immm}
\textnormal{H}:\mathcal{O}\mathfrak{p}_{\Sigma}(\mathrm{G})\rightarrow \mathcal{F}_{\Sigma}^{\star}(\mathrm{G})
\end{align}
at the given $\mathrm{G}$-oper with the induced map in hyper-cohomology for a suitable morphism $\mathcal{A}^{\bullet}\rightarrow \mathcal{B}^{\bullet},$ where $\mathcal{B}^{\bullet}$ is the holomorphic De-Rham complex of the holomorphic flat bundle $(E_{\mathrm{G}}, \omega, X).$ Working in the Dolbeault resolution, a differential-geometric calculation shows that the resulting map between hyper-cohomology groups is complex-linear and injective.  

We now mention the translation of this result to the space of $\mathrm{G}$-valued homomorphisms of the fundamental group of $\Sigma.$  Let $\widetilde{\Sigma}\rightarrow \Sigma$ be a fixed universal covering of $\Sigma$ with deck group $\pi.$  For our purposes, the deck group $\pi,$ which is abstractly isomorphic to the fundamental group of $\Sigma,$ is a more convenient model.

Taking the holonomy of a flat connection yields a bi-holomorphism
\begin{align}\label{hol iso}
\textnormal{hol}: \mathcal{F}_{\Sigma}^{\star}(\mathrm{G})\rightarrow \textnormal{Hom}^{\star}(\pi, G)/G,
\end{align}
where $\textnormal{Hom}^{\star}(\pi, G)/\mathrm{G}$ is the space of conjugacy classes of irreducible\footnote{A homomorphism $\rho: \pi\rightarrow \mathrm{G}$ is irreducible if the image of $\rho$ does not lie in any proper parabolic subgroup $P<G.$} homomorphisms $\rho: \pi\rightarrow \mathrm{G}$ with centralizer equal to the center of $G.$  Therefore, Theorem \ref{immersion} in combination with \eqref{hol iso} implies that the map
\begin{align}
\textnormal{hol}\circ \textnormal{H}: \mathcal{O}\mathfrak{p}_{\Sigma}(\mathrm{G})\rightarrow 
\textnormal{Hom}^{\star}(\pi, G)/G
\end{align}
is a holomorphic immersion.

Moving to symplectic geometry, Theorem \ref{immersion} equips the moduli space of $\Sigma$-marked $\mathrm{G}$-opers $\mathcal{O}\mathfrak{p}_{\Sigma}(\mathrm{G})$ with a closed, holomorphic $2$-form of constant rank defined as the pull-back of the Atiyah-Bott-Goldman (see \cite{GOL84}) symplectic form on $\mathcal{F}_{\Sigma}^{\star}(\mathrm{G})$ via the holomorphic immersion \eqref{immm}. 

\begin{theorem}\label{pre sym}
Let $\mathrm{G}$ be a complex simple Lie group of adjoint type.
The space $\mathcal{O}\mathfrak{p}_{\Sigma}(\mathrm{G})$ admits a closed holomorphic $2$-form $\tau_{\mathrm{G}}$ of constant rank for which the fibers of the projection to $\mathcal{T}_{\Sigma}$ are maximal isotropic sub-manifolds.
\end{theorem}

A closed holomorphic $2$-form of constant rank on a complex manifold is called a complex pre-symplectic form, so Theorem \ref{pre sym} equips $\mathcal{O}\mathfrak{p}_{\Sigma}(\mathrm{G})$ with a complex pre-symplectic form.  

We finish the discussion with the following theorem concerning the behavior of the isomorphism in Theorem \ref{id h base} with respect to the pre-symplectic structure on $\mathcal{O}\mathfrak{p}_{\Sigma}(\mathrm{G}).$

\begin{theorem}
There is a complex pre-symplectic form $\omega_{\mathcal{B}_{G}}$ on $\mathcal{B}_{\Sigma}(\mathrm{G})$ such that, given any holomorphic Lagrangian section $s$ of
\begin{align}
\mathcal{CP}_{\Sigma} \rightarrow \mathcal{T}_{\Sigma},
\end{align}
the induced bi-holomorphism
\begin{align}
\phi_{s}: \mathcal{O}\mathfrak{p}_{\Sigma}(\mathrm{G})\rightarrow \mathcal{B}_{\Sigma}(\mathrm{G})
\end{align}
satisfies $\phi_{s}^{\star}\omega_{\mathcal{B}_{G}}=\sqrt{-1}\tau_{\mathrm{G}},$ where $\tau_{\mathrm{G}}$ is the complex pre-symplectic form from Theorem \ref{pre sym}.
\end{theorem}
This generalizes a result of Kawai \cite{KAW96}, which was later clarified by 
Loustau \cite{LOU15}, to the setting of $\mathrm{G}$-opers.

Lastly, we remark that all of the constructions in this paper are invariant under the mapping class group $\textnormal{Mod}(\Sigma)$ of isotopy classes of orientation-preserving diffeomorphisms of $\Sigma.$  In particlar, there is a holomorphic action of $\textnormal{Mod}(\Sigma)$ on $\mathcal{O}\mathfrak{p}_{\Sigma}(\mathrm{G})$ lifting the usual action on $\mathcal{T}_{\Sigma}$, and the \emph{holonomy map} to $\mathcal{F}_{\Sigma}^{\star}(\mathrm{G})$ is 
$\textnormal{Mod}(\Sigma)$-equivariant.  

\subsection{Conventions and content}

We close this introduction with some comments about notational conventions.

In using the dictionary between locally free sheaves and holomorphic vector bundles, calligraphic lettters denote a locally free sheaf $\mathcal{F}$ of rank $n>0$, and Roman letters denote the corresponding rank $n$-holomorphic vector bundle $F.$  

The main technical tool in this paper is the theory of hyper-cohomology of complexes of sheaves.  For readers unfamiliar with this theory, we recommend the beautiful book of Voisin \cite{VOI07} for an elementary introduction.
\subsection{Roadmap}

In Section \ref{G theory}, we rapidly review the theory of holomorphic connections on principal $\mathrm{G}$-bundles and fix notation which will be used throughout the paper.

Section \ref{g opers} serves as an introduction to $\mathrm{G}$-opers.  In particular, we give two equivalent definitions, one of which is the original definition of Beilinson-Drinfeld \cite{BD91}, and the second of which is a translation of this definition which is more in the spirit of the theory of locally homogeneous geometric structures.  After these definitions, we explain the connection between complex projective structures and $\mathrm{G}$-opers.  

In Section \ref{models}, we survey the basic structure theory of $\mathrm{G}$-opers on a fixed Riemann surface $X.$  Most importantly, we construct explicit differential-geometric models of $\mathrm{G}$-opers, upon which all of our calculations depend.  
We also review the formal side of the theory, and the results here are all essentially due to Beilinson-Drinfeld \cite{BD91}, \cite{BD05}. But, our point of view is a bit different, and we hope that it is more accessible to the differential geometrically minded reader.

In Section \ref{families}, we build the machinery which allows us to prove that the moduli space of $\mathrm{G}$-opers admits a complex manifold structure.  This includes a discussion of Kuranishi families and the infinitesimal deformation theory of $\mathrm{G}$-opers.  We close Section \ref{families} with a proof of the identifications of the moduli space of $\mathrm{G}$-opers with the bundle of Hitchin bases $\mathcal{B}_{\Sigma}(\mathrm{G}).$ 

In the final Section \ref{hol map pre}, we prove that the map from the moduli space of $\mathrm{G}$-opers to the moduli space of flat $\mathrm{G}$-bundles over $\Sigma$ is a holomorphic immersion.  Using this result, we prove that the moduli space of $\mathrm{G}$-opers admits a natural holomorphic pre-symplectic form.  Finally, we show that there is a family of identifications of the moduli space of $\mathrm{G}$-opers with $\mathcal{B}_{\Sigma}(\mathrm{G})$ that are complex pre-symplectomorphisms for a natural complex pre-symplectic form on $\mathcal{B}_{\Sigma}(\mathrm{G}).$

\textbf{Acknowledgements}:  We deeply thank David Dumas, Bill Goldman, Brice Loustau and Richard Wentworth for many years of conversations. Furthermore, we are grateful to Patrick Brosnan for explaining how to use the Grothendieck-Riemann-Roch theorem.

\section{Gauge theory preliminaries}\label{G theory}
We begin with a rapid introduction to holomorphic flat $\mathrm{G}$-bundles and holomorphic reductions of structure.

For the purposes of these definitions, $\mathrm{G}$ may be taken to be any complex Lie group with Lie algebra $\mathfrak{g}$.  Let $E_{\mathrm{G}}$ be a holomorphic, right principal $\mathrm{G}$-bundle over a Riemann surface $X.$  For $g\in G,$ let $R_{g}: E_{\mathrm{G}}\rightarrow E_{\mathrm{G}}$ denote the holomorphic right $\mathrm{G}$-action.

\begin{definition}
A holomorphic connection on $E_{\mathrm{G}}$ is a holomorphic $1$-form
\begin{align}
\omega: TE_{\mathrm{G}}\rightarrow \mathfrak{g},
\end{align}
which satisfies:
\begin{enumerate}
\item $R_{g}^{\star}\omega=\textnormal{Ad}(g^{-1})\circ \omega$
\item If $X\in \mathfrak{g}$ and $X^{\sharp}$ is the $\mathrm{G}$-invariant vertical vector field on $E_{\mathrm{G}}$ induced by the infinitesimal $\mathrm{G}$-action, then $\omega(X^{\sharp})=X.$
\end{enumerate}
\end{definition}

If $L: G\rightarrow \textnormal{GL}[V]$ is a holomorphic representation of $\mathrm{G}$ on a complex vector space $V$, we denote the associated holomorphic vector bundle by $E_{G}[V].$ A $V$-valued holomorphic differential $k$-form $\overline{\beta}$ on $E_{\mathrm{G}}$ is called $\mathrm{G}$-equivariant
if
\begin{align}
R_{g}^{\star}\overline{\beta}=L(g^{-1})\circ \overline{\beta}
\end{align}
for all $g\in G.$

The $k$-form $\overline{\beta}$ is \emph{horizontal} if the interior product with any vertical tangent vector $Y^{\sharp}$ on $E_{\mathrm{G}}$ satisfies 
\begin{align}
\iota_{\sharp}\circ \overline{\beta}=0.
\end{align}
Given any $\mathrm{G}$-equivariant, horizontal holomorphic $1$-form $\overline{\beta}$, there exists a unique 
$\beta\in \textnormal{H}^{0}(X,\mathcal{K}\otimes \mathcal{E}_{G}[V])$ whose pullback to $E_{\mathrm{G}}$ is equal to $\overline{\beta}.$  Here, $\mathcal{K}$ is the canonical sheaf of germs of holomorphic $1$-forms on $X.$ Throughout this article, we will implicitly identify horizontal, equivariant forms $\overline{\beta}$ with their basic companion $\beta$.

The curvature of a holomorphic connection $\omega$ is the horizontal, $\mathrm{G}$-equivariant holomorphic $2$-form
\begin{align}
F(\omega):=d\omega+\frac{1}{2}[\omega,\omega].
\end{align}
In the above definition, the bracket $[\omega, \omega]$ is the result of tensoring the wedge product of $1$-forms with the Lie bracket on $\mathfrak{g}.$  The curvature descends to a global section $F(\omega)\in \textnormal{H}^{0}(X, \Omega_{X}^{2}\otimes \mathcal{E}_{\mathrm{G}}[\mathfrak{g}]).$  A holomorphic connection $\omega$ is \emph{flat} if $F(\omega)=0.$

Since $X$ is a Riemann surface, the vanishing of any holomorphic $2$-form on $X$ implies that a holomorphic connection $\omega$ on a holomorphic principal $\mathrm{G}$-bundle over $X$ is automatically flat.  

If $H<\mathrm{G}$ is a closed complex Lie subgroup and $E_{H}$ is a holomorphic reduction of structure of the bundle $E_{\mathrm{G}}$ to the subgroup $H<G,$ then the composition
\begin{align}
TE_{H}\rightarrow TE_{\mathrm{G}} \xrightarrow{\omega} \mathfrak{g}\rightarrow \mathfrak{g}/\mathfrak{h}
\end{align}
yields
  a horizontal, $H$-equivariant $1$-form
\begin{align}
\overline{\Psi}: TE_{H}\rightarrow \mathfrak{g}/\mathfrak{h}.
\end{align}
Since $\overline{\Psi}$ is horizontal and $H$-equivariant, there is a unique global section $\Psi\in \textnormal{H}^{0}(X, \mathcal{K}\otimes \mathcal{E}_{H}[\mathfrak{g}/\mathfrak{h}])$ whose pullback to $E_{H}$ agrees with $\overline{\Psi}.$  The section $\Psi$ is called the \emph{second fundamental form} of $\omega$ relative to the reduction $E_{H}.$  

Next, let $\mathcal{O}\subset \mathfrak{g}/\mathfrak{h}$ be an $H\times \mathbb{C}^{*}$-invariant subset.  

\begin{definition}\label{con position}
Let $E_{H}$ be a holomorphic reduction to a subgroup $H<\mathrm{G}$ of a holomorphic flat bundle $(E_{\mathrm{G}}, \omega).$ Then the we say that $\omega$ has relative position $\mathcal{O}$, written $\textnormal{pos}_{E_{H}}(\omega)=\mathcal{O},$ if for all non-zero tangent vectors $v\in TX,$
the second fundamental form satisfies $\Psi(v)\in E_{H}[\mathcal{O}].$
\end{definition} 
Let $\mathrm{G}$ be a connected complex semi-simple Lie group.  Given a holomorphic principal $\mathrm{G}$-bundle $E_{\mathrm{G}}$ over a complex manifold $M$ equipped with a holomorphic connection $\omega,$ the pair $(E_{\mathrm{G}}, \omega)$ is \textit{irreducible} if for every proper parabolic subgroup $P<\mathrm{G}$ and every holomorphic reduction $E_{P}$ of $E_{\mathrm{G}}$, the second fundamental form of $\omega$ relative to $E_{P}$ is non-vanishing.

\section{$\mathrm{G}$-opers}\label{g opers}

In this section, we define $\mathrm{G}$-opers on a Riemann surface $X$ where $\mathrm{G}$ is a connected complex semi-simple Lie group.  

\subsection{Lie theory preliminaries}

Before moving forward to the definiton of $\mathrm{G}$-opers, we need some Lie-theoretic preliminaries.  Choose a Borel subgroup $\mathrm{B}<\mathrm{G}$ and the corresponding Lie sub-algebra $\mathfrak{b}<\mathfrak{g}.$  Furthermore, choose a Cartan subgroup $\mathrm{H}<\mathrm{B}.$

There is an $\mathrm{H}$-invariant Lie algebra grading
\begin{align}\label{grading}
\mathfrak{g}\simeq \bigoplus_{i=-K}^{K}\mathfrak{g}_{i}
\end{align}
called the \emph{height grading}.

The height grading \eqref{grading} defines a $\mathrm{B}$-invariant filtration
\begin{align}
\mathfrak{g}^{K}\subset \mathfrak{g}^{K-1}\subset ...\subset \mathfrak{g}^{0}\subset \mathfrak{g}^{-1}\subset...\subset \mathfrak{g}^{-K}=\mathfrak{g}
\end{align}
with
\begin{align}
\mathfrak{g}^{j}=\bigoplus_{i=j}^{K} \mathfrak{g}_{i}
\end{align}
for $-K\leq j\leq K.$  In particular, $\mathfrak{g}^{0}=\mathfrak{b}.$  

The induced filtration 
\begin{align}\label{g filtration}
\mathfrak{g}^{-1}/\mathfrak{b}\subset\mathfrak{g}^{-2}/\mathfrak{b}\subset...\subset \mathfrak{g}/\mathfrak{b}
\end{align}
is $\mathrm{B}$-invariant and independent of the choice of Cartan sub-algebra.

In terms of the associated flag variety $\mathrm{G}/\mathrm{B}$, there is a $\mathrm{G}$-equivariant isomomorphism
\begin{align}\label{t bundle iso}
T(\mathrm{G}/\mathrm{B})\simeq G\times_{B} \mathfrak{g}/\mathfrak{b},
\end{align}
where $T(\mathrm{G}/\mathrm{B})$ is the tangent bundle of $\mathrm{G}/\mathrm{B}$ and $\mathrm{B}$ acts on $\mathfrak{g}/\mathfrak{b}$ via the adjoint action.

Using the isomorphism \eqref{t bundle iso}, the filtration $\eqref{g filtration}$ induces a filtration
\begin{align}\label{tan fil}
T^{-1}(\mathrm{G}/\mathrm{B})\subset T^{-2}(\mathrm{G}/\mathrm{B})\subset...\subset T(\mathrm{G}/\mathrm{B}).
\end{align}
of the tangent bundle of $\mathrm{G}/\mathrm{B}.$  

The following is an important basic fact which we will use throughout.
\begin{theorem}\label{open orbit}
There is a unique dense, open $\mathrm{B}$-orbit $\mathcal{O}\subset \mathfrak{g}^{-1}/\mathfrak{b}$ with respect to the adjoint $\mathrm{B}$-action.
\end{theorem}

\textbf{Remark:}  By the previous discussion (see \eqref{t bundle iso}), this open orbit corresponds to a sub-fiber bundle
\begin{align}\label{orbit tangent}
\mathcal{O}_{\mathrm{G}/\mathrm{B}}\subset T^{-1}(\mathrm{G}/\mathrm{B}),
\end{align}
  whose fibers are open, $\mathbb{C}^{\star}$-invariant subsets of 
the vector bundle $T^{-1}(\mathrm{G}/\mathrm{B}).$  

\begin{definition}\label{imm position}
Let $Y$ be a Riemann surface and $f: Y\rightarrow \mathrm{G}/\mathrm{B}$ a holomorphic immersion.  Then we say that $f$ has position $\mathcal{O}$ if for all non-zero vectors $v\in TY$
the differential satisfies $df(v)\in \mathcal{O}_{\mathrm{G}/\mathrm{B}}.$
\end{definition}

\subsection{$\Sigma$-marked $\mathrm{G}$-opers}
Let $\Sigma$ be a closed, connected, oriented smooth surface of genus at least two.  A $\Sigma$-marked Riemann surface $X$ is a pair $X:=(\Sigma, J)$ where $J$ is a complex structure on $\Sigma$ which induces the ambient orientation of $\Sigma.$ 

For the following, recall Definition \ref{con position}, the open orbit from Theorem \ref{open orbit}, and fix a Borel subgroup $\mathrm{B}<\mathrm{G}.$ 

\begin{definition}
A $\Sigma$-marked $\mathrm{G}$-oper is a $4$-tuple $(E_{\mathrm{G}}, E_{\mathrm{B}}, \omega, X)$ where
\begin{enumerate}
\item X is a $\Sigma$-marked Riemann surface.
\item $(E_{\mathrm{G}},\omega)$ is a holomorphic flat $\mathrm{G}$-bundle on $X.$
\item $E_{\mathrm{B}}$ is a holomorphic reduction of $E_{\mathrm{G}}$ to $\mathrm{B}<\mathrm{G}.$
\item The relative position of $\omega$ satisfies $\textnormal{pos}_{E_{\mathrm{B}}}(\omega)=\mathcal{O}.$
\end{enumerate}
\end{definition}

Now we introduce the notion of isomorphism of $\Sigma$-marked $\mathrm{G}$-opers.

\begin{definition}
Let $(E_{\mathrm{G}}, E_{\mathrm{B}}, \omega, X)$ and $(F_{\mathrm{G}}, F_{\mathrm{B}}, \eta, Y)$ be a pair of $\Sigma$-marked $\mathrm{G}$-opers.  An isomorphism is a Cartesian diagram
\[
\begin{tikzcd}
E_{\mathrm{G}}\arrow{r}{\phi} \arrow{d} &
F_{\mathrm{G}} \arrow{d} \\
X \arrow{r}{f} &
Y,
\end{tikzcd}
\]
where $\phi$ is an isomorphism of holomorphic principal $\mathrm{G}$-bundles satisfying $\phi^{\star}\eta=\omega$ and
\begin{align}
\phi |_{E_{\mathrm{B}}}: E_{\mathrm{B}}\rightarrow F_{\mathrm{B}}
\end{align}
is an isomorphism of principal $\mathrm{B}$-bundles.  Furthermore, $f:X\rightarrow Y$ is a biholomorphism whose underlying smooth map $f: \Sigma\rightarrow \Sigma$ is isotopic to the identity.
\end{definition}

This defines the category/groupoid $\widetilde{\mathcal{O}\mathfrak{p}}_{\Sigma}(\mathrm{G})$\footnote{As we have defined it, the collection of objects in this category is not a set.  This could be remedied in various ways, e.g. by working in a Grothendieck universe, or by restricting the underlying sets of the principal bundles appearing.  We will make no further mention of this issue.} of $\Sigma$-marked $\mathrm{G}$-opers.  We denote by $\mathcal{O}\mathfrak{p}_{\Sigma}(\mathrm{G})$ the set of isomorphism classes of $\Sigma$-marked $\mathrm{G}$-opers.  We will call $\mathcal{O}\mathfrak{p}_{\Sigma}(\mathrm{G})$ the moduli space of $\Sigma$-marked $\mathrm{G}$-opers.  

The Teichm\"{u}ller groupoid $\widetilde{\mathcal{T}}_{\Sigma}$ is the category whose objects are $\Sigma$-marked Riemann surfaces $X$ and morphisms $f: X\rightarrow Y$ are biholomorphisms whose underlying $C^{\infty}$-map $f:\Sigma\rightarrow \Sigma$ is isotopic to the identity.  There is an obvious full functor\footnote{We will see later that this functor is faithful if and only if $\mathrm{G}$ is of adjoint type.}
\begin{align}
\widetilde{\pi}: \widetilde{\mathcal{O}\mathfrak{p}}_{\Sigma}(\mathrm{G})\rightarrow \widetilde{\mathcal{T}}_{\Sigma},
\end{align}
which descends to a set map
\begin{align}
\pi: \mathcal{O}\mathfrak{p}_{\Sigma}(\mathrm{G})\rightarrow \mathcal{T}_{\Sigma},
\end{align}
where $\mathcal{T}_{\Sigma}$ is the set of isomorphism classes of the groupoid $\widetilde{\mathcal{T}}_{\Sigma}.$  

The set $\mathcal{T}_{\Sigma}$ is the Teichm\"{u}ller space of $\Sigma.$  It is a classical fact that $\mathcal{T}_{\Sigma}$ has the structure of a Hausd\"orff complex manifold of dimension $3g-3$ where $\mathrm{G}$ is the genus of $\Sigma.$  

We now introduce the equivalent notion of a $\Sigma$-marked \emph{developed} $\mathrm{G}$-oper, which is closer in spirit to the definition of a complex projective structure.
Fix once and for all a universal cover $\widetilde{\Sigma}\rightarrow \Sigma$ and denote the corresponding group of deck transformations by $\pi.$  With respect to our definitions, given a $\Sigma$-marked Riemann surface $X,$ this gives a unique isomorphism of the group of holomorphic deck transformations of the universal covering $\widetilde{X}\rightarrow X$ with $\pi.$  In what follows, we supress this unique identification.

In the following, recall Definition \ref{imm position}.

\begin{definition}
A $\Sigma$-marked developed $\mathrm{G}$-oper is a triple $(f, \rho, X)$ where 
\begin{enumerate}
\item X is a $\Sigma$-marked Riemann surface.
\item $\rho:\pi\rightarrow \mathrm{G}$ is a homomorphism.
\item $f: \widetilde{X}\rightarrow \mathrm{G}/\mathrm{B}$ is a holomorphic immersion satisfying $\textnormal{pos}(f)=\mathcal{O}.$
\end{enumerate}
\end{definition}

Next comes the definition of an isomorphism of $\Sigma$-marked developed $\mathrm{G}$-opers.

\begin{definition}
An isomorphism of $\Sigma$-marked developed $\mathrm{G}$-opers $(f_{1}, \rho_{1}, X_{1})$ and $(f_{2}, \rho_{2}, X_{2})$ is a commutative diagram
\[
\begin{tikzcd}
\widetilde{X}_{1} \arrow{r}{f_{1}} \arrow{d}{\widetilde{h}} &
\mathrm{G}/\mathrm{B} \arrow{d}{L_{g}} \\
\widetilde{X}_{2} \arrow{r}{f_{2}} &
\mathrm{G}/\mathrm{B}.
\end{tikzcd}
\]
where
\begin{itemize}
\item The map $\widetilde{h}: \widetilde{X}_{1}\rightarrow \widetilde{X}_{2}$ is a $\pi$-equivariant biholomorphism such that the induced underlying $C^{\infty}$-map $h: \Sigma\rightarrow \Sigma$ is isotopic to the identity.
\item  The right vertical arrow $L_{g}: \mathrm{G}/\mathrm{B} \rightarrow \mathrm{G}/\mathrm{B}$ is left translation by an element $g\in G.$
\end{itemize}
\end{definition}

As before, this defines a groupoid $\widetilde{\mathcal{DO}}\mathfrak{p}_{\Sigma}(\mathrm{G})$ of $\Sigma$-marked developed $\mathrm{G}$-opers, and the corresponding \emph{moduli space} $\mathcal{DO}\mathfrak{p}_{\Sigma}(\mathrm{G})$ of isomorphism classes of $\Sigma$-marked developed $\mathrm{G}$-opers. Note the following fact which comes immediately from the definition.

\begin{proposition}
Suppose $((f_{1}, \rho_{1}, X_{1})$ and $(f_{2}, \rho_{2}, X_{2})$ are isomorphic $\Sigma$-marked developed $\mathrm{G}$-opers.  Then there exists $g\in \mathrm{G}$ such that $\rho_{1}=g\circ \rho_{2}\circ g^{-1}.$
\end{proposition}

The next result is the promised equivalence between our first definition of $\Sigma$-marked $\mathrm{G}$-opers and the latter notion of $\Sigma$-marked developed $\mathrm{G}$-opers: we omit the proof since it is a standard exercise in differential geometry.

\begin{theorem}\label{dop to op equivalence}
There is an equivalence of categories 
\begin{align}
\widetilde{\mathcal{D}}:\widetilde{\mathcal{DO}}\mathfrak{p}_{\Sigma}(\mathrm{G}) \rightarrow \widetilde{\mathcal{O}\mathfrak{p}}_{\Sigma}(\mathrm{G})  
\end{align}
which descends to 
a bijection
\begin{align}
\mathcal{D}:\mathcal{DO}\mathfrak{p}_{\Sigma}(\mathrm{G})\rightarrow \mathcal{O}\mathfrak{p}_{\Sigma}(\mathrm{G})
\end{align}
\end{theorem}

\textbf{Remark:}  There is an obvious action of the group $\textnormal{Diff}^{+}(\Sigma)$ of orientation preserving diffeomorphisms of $\Sigma$ on these categories, and the normal subgroup $\textnormal{Diff}_{0}(\Sigma)$ of diffeomorphisms isotopic to the identity acts via isomorphisms.  Therefore, the map $\mathcal{D}$ is equivariant for the induced actions of the mapping class
group $\textnormal{Mod}(\Sigma):=\textnormal{Diff}^{+}(\Sigma)/\textnormal{Diff}_{0}(\Sigma)$ on the corresponding moduli spaces.
Finally, we quickly describe how these equivalent categories are enhancements of the usual equivalence of categories between $C^{\infty}$-flat $\mathrm{G}$-bundles on $\Sigma$ and homomorphisms $\pi\rightarrow G.$  

Let $\widetilde{\mathcal{F}}_{\Sigma}(\mathrm{G})$ be the category of $C^{\infty}$-flat $\mathrm{G}$-bundles on $\Sigma$ and $\textnormal{Hom}(\pi, \mathrm{G})$ the set of $\mathrm{G}$-valued homomorphisms of the group $\pi.$  As in Theorem \ref{dop to op equivalence} there is an equivalence of categories
\begin{align}
\widetilde{\textnormal{hol}}:\textnormal{Hom}(\pi, \mathrm{G}) \rightarrow \widetilde{\mathcal{F}}_{\Sigma}(\mathrm{G}) ,
\end{align}
where we view $\textnormal{Hom}(\pi, \mathrm{G})$ as a transformation groupoid (hence a category) under the action of $\mathrm{G}$ by conjugation.  

Consider the forgetful functors
\begin{align}\label{forget 1}
\widetilde{\mathcal{DO}}\mathfrak{p}_{\Sigma}(\mathrm{G})&\rightarrow \textnormal{Hom}(\pi, \mathrm{G}) \\
(f, \rho, X) &\mapsto \rho,
\end{align}
and 
\begin{align}\label{forget 2}
\widetilde{\mathcal{O}\mathfrak{p}}_{\Sigma}(\mathrm{G})&\rightarrow \widetilde{\mathcal{F}}_{\Sigma}(\mathrm{G}) \\
(E_{\mathrm{G}}, E_{\mathrm{B}}, \omega, X) &\mapsto (E_{\mathrm{G}}, \omega),
\end{align}
where $\omega$ is the flat $C^{\infty}$-connection on the underlying $C^{\infty}$-bundle $E_{\mathrm{G}}$ which is induced by $\omega$ and the holomorphic structure of $E_{\mathrm{G}}.$  

The proof of the following theorem is straightforward.

\begin{theorem}\label{comm diagram opers}
There is a commutative diagram
\begin{center}
\begin{tikzcd}
\widetilde{\mathcal{DO}\mathfrak{p}}_{\Sigma}(\mathrm{G})  \arrow{d} \arrow{r}{\widetilde{\mathcal{D}}}
& \widetilde{\mathcal{O}}\mathfrak{p}_{\Sigma}(\mathrm{G}) \arrow{d} \\
\textnormal{Hom}(\pi, \mathrm{G}) \arrow{r}{\widetilde{\textnormal{hol}}}
& \widetilde{\mathcal{F}}_{\Sigma}(\mathrm{G}),
\end{tikzcd}
\end{center}
where the horizontal arrows are the previously constructed equivalences of categories and the vertical arrows are the functors defined by $\eqref{forget 1}$ and $\eqref{forget 2}.$  
\end{theorem}

\subsection{Opers for $\textnormal{PSL}_{2}(\mathbb{C})$} \label{sl2 opers}

As we have promised to exhibit $\mathrm{G}$-opers as a generalization of complex projective structures on Riemann surfaces, in this section we recall the basic properties of the space of complex projective structures and elucidate the relationship to $\mathrm{G}$-opers.

A complex projective structure on $\Sigma$ is a maximal atlas of charts with values in $\mathbb{CP}^{1}$ whose transition maps are given by restrictions of M\"obius transformations.  In particular, any complex projective structure induces a Riemann surface structure on $\Sigma.$  We refer to such a structure as a $\Sigma$-marked complex projective structure.

Given two $\Sigma$-marked complex projective structures $Z_{1}$ and $Z_{2},$ an isomorphism is a smooth map
\begin{align}
h: Z_{1}\rightarrow Z_{2}
\end{align}
whose projective coordinate representation is locally given by a M\"obius transformation, and such that the underlying smooth map $h: \Sigma\rightarrow \Sigma$ is isotopic to the identity.  The moduli space of $\Sigma$-marked complex projective structures $\mathcal{CP}_{\Sigma}$ is the set of isomorphism classes of $\Sigma$-marked complex projective structures.

Equivalently, given a $\Sigma$-marked complex projective structure $Z,$ lifting the structure to the universal cover $\widetilde{Z}$, the coordinate charts globalize to a holomorphic immersion
\begin{align}
f: \widetilde{Z}\rightarrow \mathbb{CP}^{1},
\end{align}
which is equivariant for a homomorphism $\rho: \pi\rightarrow \textnormal{PSL}_{2}(\mathbb{C}).$ 

For $G=\textnormal{PSL}_{2}(\mathbb{C}),$ the flag variety $\mathrm{G}/\mathrm{B}$ is $\mathrm{G}$-equivariantly isomorphic to the complex projective line $\mathbb{CP}^{1}.$  In this case, the definition of a $\Sigma$-marked developed $\mathrm{G}$-oper yields a triple $(f,\rho, X)$ consisting of a $\Sigma$-marked Riemann surface $X,$ a homomorphism $\rho: \pi \rightarrow \textnormal{PSL}_{2}(\mathbb{C})$, and a $\rho$-equivariant local bi-holomorphism
$f: \widetilde{X}\rightarrow \mathbb{CP}^{1}.$

Therefore, a $\Sigma$-marked developed $\textnormal{PSL}_{2}(\mathbb{C})$-oper $(f,\rho, X)$ is the same as a $\Sigma$-marked complex projective structure; or in other terminology, a locally homogeneous $(\textnormal{PSL}_{2}(\mathbb{C}), \mathbb{CP}^{1})$ geometric structure on $\Sigma$ in the sense of Ehressmann-Thurston. 

Let $(f_{1},\rho_{1}, X), (f_{2}, \rho_{2}, X)$ be two $\Sigma$-marked developed $\textnormal{PSL}(2,\mathbb{C})$-opers.  Choose a small open set $U\in \widetilde{X}$ such that $f_{1}|_{U}$ is a biholomorphism onto $f_{1}(U).$
Then,
\begin{align}
f_{2} \circ f_{1}^{-1}: f_{1}(U)\rightarrow \mathbb{CP}^{1}
\end{align}
is a locally injective holomorphic map.  

Given any holomorphic map $q: V\rightarrow \mathbb{CP}^{1}$ where $V\subset \mathbb{CP}^{1}$ is an open set, let $j_{x}^{k}(q)$ be the holomorphic k-jet of the map $q$ at $x\in V.$  

 The action of a M\"obius transformation $g\in \textnormal{PSL}(2,\mathbb{C})$ is denoted by $L_{g}: \mathbb{CP}^{1}\rightarrow \mathbb{CP}^{1}.$
 
 A proof of the following proposition may be found in \cite{HUB06}.

\begin{proposition}\label{osc map}
There exists a unique holomorphic map
\begin{align}
\mathrm{O}_{f_{1}, f_{2}}: \widetilde{X}\rightarrow \textnormal{PSL}_{2}(\mathbb{C})
\end{align}
which satisfies 
\begin{align}
j_{f_{1}(x)}^{2}(L_{\mathrm{O}_{f_{1},f_{2}}(x)})=j_{f_{1}(x)}^{2}(f_{2}\circ f_{1}^{-1}).
\end{align}
Furthermore, for every $\gamma\in \pi,$
\begin{align}
\mathrm{O}_{f_{1},f_{2}}(\gamma(x))=\rho_{2}(\gamma)\circ \mathrm{O}_{f_{1},f_{2}}(x) \circ \rho_{1}(\gamma^{-1}).
\end{align}
\end{proposition}

In the complex projective structures literature, the map $\mathrm{O}_{f_{1},f_{2}}$ is usually called the \emph{osculating} map.

By Proposition \ref{osc map}, the map
\begin{align}
\widetilde{X}\times \textnormal{PSL}_{2}(\mathbb{C}) &\rightarrow \widetilde{X} \times \textnormal{PSL}_{2}(\mathbb{C}) \\
(x, g) &\mapsto (x, \mathrm{O}_{f_{1},f_{2}}(x)g)
\end{align}
descends to an isomorphism of holomorphic principal $\textnormal{PSL}_{2}(\mathbb{C})$-bundles
\begin{align}
\mathrm{O}_{f_{1}, f_{2}}: \widetilde{X} \times_{\rho_{1}} \textnormal{PSL}_{2}(\mathbb{C}) \rightarrow \widetilde{X} \times_{\rho_{2}} \textnormal{PSL}_{2}(\mathbb{C}).
\end{align}

Let $\{f_{1}, x, e_{1}\}$ be a fixed $\mathfrak{sl}_{2}(\mathbb{C})$-triple in $\mathfrak{sl}_{2}(\mathbb{C})$ where the span of $\{x, e_{1}\}$ is the equal to the upper triangular Borel sub-algebra.  

Recalling the $\Sigma$-marked developed $\textnormal{PSL}_{2}(\mathbb{C})$-oper $(f_{1},\rho_{1}, X_{1}),$ the locally injective holomorphic map $f_{1}:\widetilde{X}\rightarrow \mathbb{CP}^{1}$ defines a holomorphic reduction of $\widetilde{X}\times_{\rho_{1}} \textnormal{PSL}_{2}(\mathbb{C})$ to the Borel $\mathrm{B}<\mathrm{G}:$
\begin{align}
E_{\mathrm{B}}:=\{ [(x,g)]\in \widetilde{X}\times_{\rho_{1}} \textnormal{PSL}_{2}(\mathbb{C}) \ | \ L_{g^{-1}}\circ f_{1}(x)=e\mathrm{B}\in \textnormal{PSL}_{2}(\mathbb{C})/\mathrm{B}\}.
\end{align}
Note that the definition of $E_{\mathrm{B}}$ is invariant the \emph{right} $\mathrm{B}$-action on pairs in $\widetilde{X}\times_{\rho_{1}} \textnormal{PSL}_{2}(\mathbb{C}).$  

The Borel subgroup acts on $\mathfrak{g}_{1}=\textnormal{span}_{\mathbb{C}}(e_{1}),$
and therefore the associated line bundle $E_{\mathrm{B}}[\mathfrak{g}_{1}]$ is well defined.  It is a direct 
consequence of the oper condition that $E_{\mathrm{B}}[\mathfrak{g}_{1}]\simeq K.$  

Given a $\Sigma$-marked Riemann surface $X,$ let $\widetilde{\mathcal{DO}\mathfrak{p}}_{X}(\textnormal{PSL}_{2}(\mathbb{C}))$ be the fiber over $X$ of the fully faithful functor
\begin{align}
\widetilde{\mathcal{DO}\mathfrak{p}}_{\Sigma}(\textnormal{PSL}_{2}(\mathbb{C}))\rightarrow \widetilde{\mathcal{T}}_{\Sigma},
\end{align}
and $\mathcal{DO}\mathfrak{p}_{X}(\textnormal{PSL}_{2}(\mathbb{C}))$ the corresponding set of isomorphism classes.

Then we have the following proposition (see \cite{DUM09}).
\begin{proposition}\label{param sl2}
If $\omega_{\rho_{2}}$ is the canonically defined flat connection on $\widetilde{X} \times_{\rho_{2}} \textnormal{PSL}_{2}(\mathbb{C})$ and $\omega_{\rho_{1}}$ is the canonically defined flat connection on $\widetilde{X} \times_{\rho_{1}} \textnormal{PSL}_{2}(\mathbb{C})$, then
\begin{align}
\mathrm{O}_{f_{1}, f_{2}}^{\star}(\omega_{\rho_{2}})-\omega_{\rho_{1}} \in \textnormal{H}^{0}(X, \mathcal{K}\otimes \mathcal{E}_{\mathrm{B}}[\mathfrak{g}_{1}])\simeq \textnormal{H}^{0}(X, \mathcal{K}^{2}).
\end{align}
Moreover, this assignment
defines a bijection (which depends on $(f_{1}, \rho_{1})),$
\begin{align}
\mathcal{DO}\mathfrak{p}_{X}(\textnormal{PSL}_{2}(\mathbb{C}))\simeq \textnormal{H}^{0}(X, \mathcal{K}^{2}).
\end{align}
This gives the space $\mathcal{DO}\mathfrak{p}_{X}(\textnormal{PSL}_{2}(\mathbb{C}))$ the structure of an affine space with underlying vector space of translations $\textnormal{H}^{0}(X, \mathcal{K}^{2}).$
\end{proposition}

\textbf{Remark:}  Soon, we will see that this is a general phenomenon for $\mathrm{G}$-opers when $\mathrm{G}$ is a complex simple Lie group of adjoint type.  The above description is equivalent to the classical identification of complex projective structures with holomorphic quadratic differentials that arises from the Schwarzian derivative \cite{DUM09}.  In particular, up to a constant multiple, the quadratic differential appearing in Proposition \ref{param sl2} is the Schwarzian derivative of the locally univalent (multi-valued) map $f_{2}\circ f_{1}^{-1}.$

The question of understanding the moduli space of $\Sigma$-marked complex projective structures was analyzed by Hubbard in \cite{HUB81} where he proved the following theorem.
We adopt the language of opers here, whereas Hubbard worked directly with complex projectives structures since opers were not defined at the time.

\begin{theorem}\label{Hubbard}
The moduli space $\mathcal{O}\mathfrak{p}_{\Sigma}(\textnormal{PSL}_{2}(\mathbb{C}))$ of $\Sigma$-marked $\textnormal{PSL}_{2}(\mathbb{C})$-opers has the structure of a complex manifold of dimension $6g-6$ where $g$ is the genus of $\Sigma.$
Furthermore, the map
\begin{align}\label{cps proj}
\mathcal{O}\mathfrak{p}_{\Sigma}(\textnormal{PSL}_{2}(\mathbb{C})) \rightarrow \mathcal{T}_{\Sigma}
\end{align}
is a holomorphic affine bundle.

Finally, every holomorphic section of the projection
\begin{align}
\mathcal{O}\mathfrak{p}_{\Sigma}(\textnormal{PSL}_{2}(\mathbb{C})) \rightarrow \mathcal{T}_{\Sigma}
\end{align}
induces a biholomorphism
\begin{align}
\mathcal{O}\mathfrak{p}_{\Sigma}(\textnormal{PSL}_{2}(\mathbb{C})) \simeq T^{*}\mathcal{T}_{\Sigma}
\end{align}
where $T^{*}\mathcal{T}_{\Sigma}$ is the cotangent bundle of Teichm\"{u}ller space.
\end{theorem}

\textbf{Remark:}  There is a family of Bers' holomorphic sections of the bundle \eqref{cps proj}, each of which  yields a holomorphic identification
\begin{align}
\mathcal{O}\mathfrak{p}_{\Sigma}(\textnormal{PSL}_{2}(\mathbb{C}))\simeq T^{*}\mathcal{T}_{\Sigma}.
\end{align}
A Bers' section requires the choice of a (conjugate) Riemann surface $[\overline{Y}]\in \overline{\mathcal{T}_{\Sigma}}$, and is defined by sending a Riemann surface $[X]\in \mathcal{T}_{\Sigma}$ to the complex projective structure on the \emph{top} component of the quasi-Fuchsian manifold
determined by $(X, \overline{Y})\in \mathcal{T}_{\Sigma}\times \overline{\mathcal{T}_{\Sigma}}.$
In section \ref{families}, we will prove a generalization of Theorem \ref{Hubbard} exhibiting the space of $\mathrm{G}$-opers (for $\mathrm{G}$ complex simple of adjoint type) as a bundle over $\mathcal{T}_{\Sigma}.$

\section{Explicit Models}\label{models}

In this short section, we will give an explicit construction of $\mathrm{G}$-opers where $\mathrm{G}$ is a complex simple Lie group of adjoint type.  These explicit models will be essential for many cohomological calculations we make later in the paper.  

Let $\mathfrak{g}$ be the simple Lie algebra of $\mathrm{G}$ and $\ell$ denote the rank of $\mathfrak{g}.$  As always, we have fixed a Borel subalgebra $\mathfrak{b}<\mathfrak{g}.$  Recall the height grading 
\begin{align}\label{height}
\mathfrak{g}\simeq \bigoplus_{i=-m_{\ell}}^{m_{\ell}} \mathfrak{g}_{i}
\end{align}
corresponding to a Cartan subalgebra $\mathfrak{h}<\mathfrak{b}.$ 

Fixing $e_{1}\in \mathfrak{g}_{1}$ a principal nilpotent element (i.e. an element whose projection onto every simple root space is non-vanishing), the centralizer of $e_{1}$ is an $\ell$-dimensional subspace; we fix basis vectors $e_{i}\in \mathfrak{g}_{m_{i}}$
such that $\{e_{1},...,e_{\ell}\}\subset \mathfrak{g}$ spans the centralizer of $e_{1}.$  The numbers $1=m_{1}\leq...\leq m_{\ell}$ are the exponents of $\mathfrak{g}.$  

Next, complete the regular nilpotent element $e_{1}$ to an $\mathfrak{sl}_{2}$-triple $\{f_{1}, x, e_{1}\}\subset \mathfrak{g}.$.  Then we have
$f_{1}\in \mathfrak{g}_{-1}, x\in \mathfrak{g}_{0}, e_{1}\in \mathfrak{g}_{1}.$  Therefore, since \eqref{height} is a grading, 
\begin{align}\label{f1 action}
\textnormal{ad}(f_{1})(\mathfrak{g}_{i})\subset \mathfrak{g}_{i-1},
\end{align}
and 
\begin{align}\label{e1 action}
\textnormal{ad}(e_{1})(\mathfrak{g}_{i})\subset \mathfrak{g}_{i+1}.
\end{align}
\textbf{Remark}
The sub-algebra generated by an $\mathfrak{sl}_{2}$-triple $\{f_{1}, x, e_{1}\}\subset \mathfrak{g}$ with $e_{1}$ principal nilpotent is called a principal three-dimensional sub-algebra.

The $\mathfrak{sl}_{2}$-triple $\{f_{1}, x, e_{1}\}\subset \mathfrak{g}$ induces an injective homomorphism
\begin{align}
\iota_{\mathrm{G}}: \textnormal{PSL}_{2}(\mathbb{C})\rightarrow G
\end{align}
called the principal three-dimensional subgroup.  Moreover, there is a holomorphic embedding
\begin{align}
f_{G}: \mathbb{CP}^{1}\rightarrow \mathrm{G}/\mathrm{B}
\end{align}
which is $\iota_{\mathrm{G}}$-equivariant.  The holomorphic map $f_{G}$ is called the principal rational curve.

Fix a $\Sigma$-marked Riemann surface $X$ and let $K^{i}$ be the $i$-th tensor power of the canonical bundle of $X.$  

Consider the $C^{\infty}$-vector bundle
\begin{align}
E_{\mathrm{G}}[\mathfrak{g}]:=\bigoplus_{i=-m_{\ell}}^{m_{\ell}} K^{i}\otimes \mathfrak{g}_{i}
\end{align}
with structure group $\mathrm{G},$ where the grading of $\mathfrak{g}$ is taken from \eqref{height}.

Given a smooth section of $E_{\mathrm{G}}[\mathfrak{g}]$
\begin{align}
s=\sum_{i} \beta_{i}\otimes V_{i},
\end{align}
where $\beta_{i}\otimes V_{i}\in \mathcal{A}^{0}(X, K^{i}\otimes \mathfrak{g}_{i}),$ define a holomorphic structure on $E_{\mathrm{G}}[\mathfrak{g}]$ via the $\overline{\partial}$-operator:
\begin{align}
\overline{\partial} s=\sum_{i} \overline{\partial}_{i} \beta_{i}\otimes V_{i}+ h\cdot \beta_{i}\otimes \textnormal{ad}(e_{1})(V_{i}).
\end{align}
Above, $\overline{\partial}_{i}$ is the $\overline{\partial}$-operator on $K^{i}$ and $h$ is the Hermitian metric on $\Theta$ arising from the unique hyperbolic metric which uniformizes $X$, which we view as a tensor $h\in \mathcal{A}^{(0,1)}(X, K).$  Furthermore, the term
$h\cdot \beta_{i}$ is viewed as a tensor in $\mathcal{A}^{(0,1)}(X, K^{i+1}).$  Finally, note that this is a well-defined $\overline{\partial}$-operator by \eqref{e1 action}.

For every non-zero $i\in [-m_{\ell}, m_{\ell}]\cap \mathbb{Z},$ the Hermitian metric $h$ induces a Hermitian metric on $K^{i}$, and for $K^{0}\simeq \mathcal{O}$ we use the trivial Hermitian metric
\begin{align}
(f, g)\mapsto f\overline{g}.
\end{align}
  
If $\partial_{i}$ is the $(1,0)$-part of the Chern connection of the induced Hermitian metric on $K^{i},$ define an operator $\nabla$ by the formula:
\begin{align}
\nabla s=\sum_{i} \partial_{i} \beta_{i}\otimes V_{i}+ \beta_{i} \otimes \textnormal{ad}(f_{1})(V_{i}).
\end{align}
Viewing $\beta_{i}\in \mathcal{A}^{(1,0)}(X, K^{i-1}),$ this is a well-defined connection of type $(1,0)$ by \eqref{f1 action}.

\begin{proposition}
The operator $\nabla$ defines a holomorphic connection on the holomorphic vector bundle $(E_{\mathrm{G}}[\mathfrak{g}], \overline{\partial})$ and the $C^{\infty}$-connection $D=\nabla+\overline{\partial}$ is flat.   
\end{proposition}

\begin{proof}
The $C^{\infty}$-connection $D=\nabla+ \overline{\partial}$ is flat if and only if
\begin{align}
F(D)=F(\nabla)+\overline{\partial}^{2}+\nabla\circ \overline{\partial} +\overline{\partial}\circ \nabla=0.
\end{align}
Because we are on a Riemann surface, $F(\nabla)=\overline{\partial}^{2}=0,$ and therefore if
\begin{align}
\nabla\circ \overline{\partial} +\overline{\partial}\circ \nabla=0,
\end{align}
we simultaneously obtain that $\nabla$ is holomorphic and $D$ is flat.

Letting
\begin{align}
s=\sum_{i} \beta_{i}\otimes V_{i},
\end{align}
we compute
\begin{align}
\nabla\circ \overline{\partial}s&=\sum_{i} \partial_{i}\overline{\partial_{i}}\beta_{i}\otimes V_{i} - \overline{\partial}_{i}\beta_{i}\otimes \textnormal{ad}(f_{1})(V_{i}) \\
&+\partial_{i+1}(h\cdot \beta_{i})\otimes \textnormal{ad}(e_{1})(V_{i}) 
- h\cdot \beta_{i}\otimes \textnormal{ad}(f_{1})\circ \textnormal{ad}(e_{1})(V_{i}).
\end{align}
There are some subtle identifications here; namely we must consider the following terms as elements of the following spaces:
\begin{enumerate}
\item $\overline{\partial}_{i}\beta_{i}\in \mathcal{A}^{(1,1)}(X, K^{i-1})$.
\item $\partial_{i+1}(h\cdot \beta_{i})\in \mathcal{A}^{(1,1)}(X, K^{i+1})$
\item $h\cdot \beta_{i}\in \mathcal{A}^{(1,1)}(X, K^{i}).$
\end{enumerate}

In the other direction, we compute
\begin{align}
\overline{\partial}\circ \nabla s&=\sum_{i} \overline{\partial}_{i}\partial_{i} \beta_{i}\otimes V_{i}
+ \overline{\partial}_{i}\beta_{i}\otimes \textnormal{ad}(f_{1})(V_{i}) \\
&-h\cdot \partial_{i}\beta_{i}\otimes \textnormal{ad}(e_{1})(V_{i}) 
+ h\cdot \beta_{i}\otimes \textnormal{ad}(e_{1})\circ \textnormal{ad}(f_{1})(V_{i}).
\end{align}
Summing these two terms yields
\begin{align}
\nabla\circ \overline{\partial}s+ \overline{\partial}\circ \nabla s&=\sum_{i} (\partial_{i}\overline{\partial_{i}}\beta_{i}+ \overline{\partial}_{i}\partial_{i} \beta_{i}) \otimes V_{i} \\
&+ h\cdot \beta_{i}\otimes [\textnormal{ad}(e_{1}), \textnormal{ad}(f_{1})](V_{i}) \\
&+ (\partial_{i+1}(h\cdot \beta_{i})- h\cdot \partial_{i}\beta_{i})\otimes \textnormal{ad}(e_{1})(V_{i}) \\
&=\sum_{i} F(\nabla^{i})\beta_{i} \otimes V_{i} +ih\cdot \beta_{i}\otimes V_{i} \\
&+ (\partial_{i+1}(h\cdot \beta_{i})- h\cdot \partial_{i}\beta_{i})\otimes \textnormal{ad}(e_{1})(V_{i}).
\end{align}
Here, $\nabla^{i}$ is the Chern connection of the Hermitian metric on $K^{i}$ induced by the Hermitian metric on $K^{-1}$ associated to the uniformizing hyperbolic metric on $X.$  
But, since the hyperbolic metric has constant curvature $-1$,
\begin{align}
F(\nabla^{i})\beta_{i}=-ih\cdot \beta_{i},
\end{align}
which implies
\begin{align}
\sum_{i} F(\nabla^{i})\beta_{i} \otimes V_{i} +ih\cdot \beta_{i}\otimes V_{i}=0.
\end{align}
Furthermore, since the metric $h$ is parallel for the Chern connection,
\begin{align}
\sum_{i} \partial_{i+1}(h\cdot \beta_{i})- h\cdot \partial_{i}\beta_{i}=0.
\end{align}
Hence,
\begin{align}
\nabla\circ \overline{\partial}s+ \overline{\partial}\circ \nabla s=0,
\end{align}
which completes the proof.
\end{proof}

The sub-bundle 
\begin{align}
E_{\mathrm{B}}[\mathfrak{b}]:=\bigoplus_{i=0}^{m_{\ell}} K^{i}\otimes \mathfrak{g}_{i}
\end{align}
is $\overline{\partial}$-invariant, and thus defines a holomorphic sub-bundle of $E_{\mathrm{G}}[\mathfrak{g}].$ The following proposition is immediate.

\begin{proposition}\label{unif oper}
Let $E_{\mathrm{G}}$ be the holomorphic principal $\mathrm{G}$-bundle whose associated adjoint bundle is $E_{\mathrm{G}}[\mathfrak{g}],$ and $E_{\mathrm{B}}$ the reduction of structure whose corresponding adjoint bundle is 
$E_{\mathrm{B}}[\mathfrak{b}].$  Finally, let $\omega$ be the holomorphic flat connection induced by $\nabla.$  Then $(E_{\mathrm{G}}, E_{\mathrm{B}}, \omega, X)$ is a $\Sigma$-marked $\mathrm{G}$-oper whose second fundamental form is given by
\begin{align}
\Psi: \Theta&\rightarrow K^{-1}\otimes \mathfrak{g}_{-1}\simeq E_{\mathrm{B}}[\mathfrak{g}^{-1}/\mathfrak{b}]\\
\xi &\mapsto \xi\otimes f_{1}.
\end{align}
\end{proposition}

\begin{proof}
The only thing to recognize is that since $f_{1}\in \mathfrak{g}_{-1}$ is a principal nilpotent element, $f_{1}\in \mathcal{O}$ where $\mathcal{O}\in \mathfrak{g}^{-1}/\mathfrak{b}$ is the unique open orbit from the proof of Theorem \ref{open orbit}.  Hence,
\begin{align}
\Psi(\xi)\in E_{\mathrm{B}}(\mathcal{O})
\end{align}
for all non-zero tangent vectors $\xi.$  This completes the proof.
\end{proof}

\textbf{Remark:} Under the correspondence with $\Sigma$-marked developed $\mathrm{G}$-opers, the $\mathrm{G}$-oper of Proposition \ref{unif oper} is the triple $(f, \rho, X)=(f_{G}\circ f_{0}, \iota_{\mathrm{G}}\circ \rho_{0}, X)$ where
\begin{align}
\rho_{0}: \pi\rightarrow \textnormal{PSL}(2, \mathbb{R})
\end{align}
is the Fuchsian homomorphism uniformizing $X$, and
\begin{align}
f_{0}: \widetilde{X}\rightarrow \mathbb{H}^{2}\subset \mathbb{CP}^{1}
\end{align}
is the developing map of the uniformizing hyperbolic structure on $X.$  

Given a tuple $\vec{\alpha}:=(\alpha_{1},...\alpha_{\ell})\in \bigoplus_{i=1}^{\ell} \textnormal{H}^{0}(X, \mathcal{K}^{m_{i}+1})$, form the new operator
\begin{align}
\nabla^{\vec{\alpha}}=\nabla + \sum_{i=1}^{\ell} \alpha_{i}\otimes \textnormal{ad}(e_{i}),
\end{align}
where we recall that $\{e_{i}\}_{i=1}^{\ell}$ is a homogeneous basis of the centralizer of $e_{1}$ with respect to the grading \eqref{height}.

Since the $\alpha_{i}$ are holomorphic and 
\begin{align}
[e_{1}, e_{i}]=0
\end{align}
for all $1\leq i \leq \ell,$ it immediately follows that 
\begin{align}
\sum_{i=1}^{\ell} \alpha_{i}\otimes e_{i}\in \textnormal{H}^{0}(X, \mathcal{K}\otimes \mathcal{E}_{\mathrm{G}}[\mathfrak{g}]).
\end{align}
Therefore, $\nabla^{\vec{\alpha}}$ is a holomorphic connection on $(E_{\mathrm{G}}[\mathfrak{g}], \overline{\partial}).$  Furthermore, the second fundamental form of $\nabla^{\vec{\alpha}}$ with respect to $E_{\mathrm{B}}[\mathfrak{b}]$ is still given by the $\Psi$ of Proposition \ref{unif oper}.

Let $\omega^{\vec{\alpha}}$ denote the corresponding holomorphic connection on the principal bundle $E_{\mathrm{G}}.$  The following proposition follows immediately from Proposition \ref{unif oper}.

\begin{proposition}\label{oper param 1}
For every $\vec{\alpha}\in \bigoplus_{i=1}^{\ell} \textnormal{H}^{0}(X, \mathcal{K}^{m_{i}+1}), \
(E_{\mathrm{G}}, E_{\mathrm{B}}, \omega^{\vec{\alpha}}, X)$ defines a $\Sigma$-marked $\mathrm{G}$-oper.
\end{proposition}
Proposition \ref{oper param 1} yields a well-defined map 
\begin{align}
\bigoplus_{i=1}^{\ell} \textnormal{H}^{0}(X, \mathcal{K}^{m_{i}+1})\rightarrow \mathcal{O}\mathfrak{p}_{X}(\mathrm{G}),
\end{align}
which depends upon the choice of a $\textnormal{PSL}_{2}(\mathbb{C})$-oper on $X$ and a homogeneous basis of the centralizer in $\mathfrak{g}$ of $e_{1}\in \mathfrak{g}_{-1}.$  In the next section, we will see that this map is a bijection.

\subsection{Parameterizing $\mathrm{G}$-opers}\label{param}

In this short section, we quickly review the parameterization of $\mathrm{G}$-opers (for $\mathrm{G}$ simple of adjoint type) on $X$ via pluri-canonical sections on $X$, see \cite{BD91} and \cite{BD05} for details.  In particular, this gives a more intrinsic construction of the explicit models from Section \ref{models}.

Let $\mathrm{G}$ be a complex simple Lie group of adjoint type and rank $\ell.$  Fix a Borel subgroup $\mathrm{B}<\mathrm{G}$ with corresponding sub-algebra $\mathfrak{b}<\mathfrak{g}.$ Let
\begin{align}
\iota_{\mathfrak{g}}: \mathfrak{sl}_{2}(\mathbb{C})\rightarrow \mathfrak{g}
\end{align}
be a principal three-dimensional sub-algebra
 sending the upper triangular Borel subalgebra in $\mathfrak{sl}_{2}(\mathbb{C})$ to the fixed Borel sub-algebra $\mathfrak{b}<\mathfrak{g}$: this uniquely defines an injective homomorphism $\iota_{\mathrm{G}}: \textnormal{PSL}_{2}(\mathbb{C})\rightarrow \mathrm{G}.$  

Let $\{f_{1}, x, e_{1}\}\subset \mathfrak{g}$ be the corresponding $\mathfrak{sl}_{2}(\mathbb{C})$-triple generating the image of $\iota_{\mathfrak{g}},$ and set $V=\textnormal{Ker}(\textnormal{ad}(e_{1})).$  Because $e_{1}\in \mathfrak{g}$ is a principal nilpotent element, the vector space $V$ is $\ell$-dimensional.  

The regular semi-simple element $x$ induces a line decomposition 
\begin{align}
V=\bigoplus_{i=1}^{\ell} V\cap \mathfrak{g}_{m_{i}}=\bigoplus_{i=1}^{\ell} V_{m_{i}},
\end{align}
 where $\{m_{1},...m_{\ell}\}$ are the exponents of $\mathfrak{g}.$  If $\mathrm{B}_{0}< \textnormal{PSL}_{2}(\mathbb{C})$ is the standard Borel of upper triangular matrices, then $V$ is $\mathrm{B}_{0}$-invariant, where $\mathrm{B}_{0}$ acts using the homomorphism $\iota_{\mathrm{G}}: \textnormal{PSL}_{2}(\mathbb{C})\rightarrow G.$  Therefore, to any principal $\mathrm{B}_{0}$-bundle $E_{\mathrm{B}_{0}},$ there is an associated vector bundle $E_{\mathrm{B}_{0}}[V].$

The following theorem due to \cite{BD05} is a generalization of Proposition \ref{param sl2}.

\begin{theorem} \label{unique iso}
Let $\mathrm{G}$ be a complex simple Lie group of adjoint type.
Let $(E_{\mathrm{G}_{0}}, E_{\mathrm{B}_{0}}, \mu, X)$ be a $\Sigma$-marked $\textnormal{PSL}_{2}(\mathbb{C})$-oper and $(E_{\mathrm{G}}, E_{\mathrm{B}}, \omega, X)$ be a $\Sigma$-marked $\mathrm{G}$-oper.  Then there exists a unique
isomorphism $\phi: E_{\mathrm{B}_{0}}[\mathrm{B}]\rightarrow E_{\mathrm{B}}$ such that $\phi^{*}\omega - \mu \in \textnormal{H}^{0}(X, \mathcal{K}\otimes \mathcal{E}_{\mathrm{B}_{0}}[V]).$
\end{theorem}

Fixing a homogeneous basis of the graded vector space $V,$ by \cite{BD05} there is canonical isomorphism $\mathcal{E}_{\mathrm{B}_{0}}[V]\simeq \bigoplus_{i=1}^{\ell} \mathcal{K}^{m_{i}}.$

This yields the promised parameterization of $\mathcal{O}\mathfrak{p}_{X}(\mathrm{G}).$

\begin{theorem}[\cite{BD05}]\label{param opers 2}
The map 
\begin{align}
\mathcal{O}\mathfrak{p}_{X}(\mathrm{G})&\rightarrow \textnormal{H}^{0}(X, \mathcal{K}\otimes \mathcal{E}_{\mathrm{B}_{0}}[V])\simeq \bigoplus_{i=1}^{\ell} \textnormal{H}^{0}(X, \mathcal{K}^{m_{i}+1}) \\
(E_{\mathrm{G}}, E_{\mathrm{B}}, \omega, X) &\mapsto   \phi^{\star}\omega-\mu
\end{align}
is a bijection.  Furthermore, upon choosing the appropriate homogeneous basis of $V,$ it is inverse to the map defined by Proposition \ref{oper param 1}
\end{theorem}

\textbf{Remark:} Note that this bijection depends upon a choice of $\textnormal{PSL}_{2}(\mathbb{C})$-oper on $X$ and a choice of homogeneous basis of the graded vector space $V.$  These are exactly the choices that we had to make to construct the explicit models of section \ref{models}.  Finally, observe that Theorem \ref{param opers 2} verifies the claim following Proposition \ref{oper param 1}, namely that the explicit models yield a parameterization of the space of $\Sigma$-marked $\mathrm{G}$-opers on $X.$  

\section{Kuranishi families and the global structure of $\Sigma$-marked $\mathrm{G}$-opers}\label{families}

This section develops the machinery to prove the existence of a natural complex structure on the space $\mathcal{O}\mathfrak{p}_{\Sigma}(\mathrm{G})$ of $\Sigma$-marked $\mathrm{G}$-opers.  The approach here is standard in (the analyic approach to) deformation theory, and consists of four parts:
\begin{enumerate}
\item Develop a notion of families of $\Sigma$-marked $\mathrm{G}$-opers.
\item Identify the infinitesimal deformations and obstructions with some cohomology groups of an appropriate complex of sheaves.
\item Apply the Kuranishi method (Hodge theory, elliptic complexes, etc.) to show that every unobstructed infinitesimal deformation is tangent to an honest deformation.
\item Use these families to construct a holomorphic atlas on $\mathcal{O}\mathfrak{p}_{\Sigma}(\mathrm{G}).$
\end{enumerate}
We will develop in detail the first two parts of this program.  The third part would require a significant detour into Hodge theory, elliptic complexes, and related analytic notions; therefore we have chosen to omit these details from the paper.  We will give references to completely analogous constructions in the literature from which the reader can, with some work, fill in the details of this technical part.  Finally, we will carry out the final process of constructing a holomorphic atlas on 
$\mathcal{O}\mathfrak{p}_{\Sigma}(\mathrm{G}).$

We begin with the following definition.

\begin{definition}
A $\Sigma$-marked family of Riemann surfaces is a tuple $(\mathcal{X}, \mathcal{B},p)$
where $\mathcal{X}$ and $\mathcal{B}$ are complex manifolds such that:
\begin{enumerate}
\item $p: \mathcal{X}\rightarrow \mathcal{B}$ is a proper holomorphic submersion.
        
\item In the $C^{\infty}$-category, $p: \mathcal{X}\rightarrow \mathcal{B}$ is a fiber bundle with fiber $\Sigma$ and structure
         group $\textnormal{Diff}_{0}(\Sigma).$
\end{enumerate}
\end{definition}

A morphism of $\Sigma$-marked families is defined as follows.

\begin{definition}
Given a pair of $\Sigma$-marked families of Riemann surfaces $(\mathcal{X}_{1}, \mathcal{B}_{1}, p_{1})$ and
$(\mathcal{X}_{2}, \mathcal{B}_{2}, p_{2})$, a morphism $F=(\phi, f)$ consists of
\begin{enumerate}
\item A Cartesian diagram
\begin{center}
\begin{tikzcd}
\mathcal{X}_{1} \arrow{r}{\phi} \arrow{d}{p_{1}}
& \mathcal{X}_{2} \arrow{d}{p_{2}} \\
\mathcal{B}_{1} \arrow{r}{f}
& \mathcal{B}_{2}
\end{tikzcd}
\end{center}
Such that $\phi$ and $f$ are holomorphic maps and the pair $(\phi, f)$, as $C^{\infty}$-maps, are maps of $(\Sigma, \textnormal{Diff}_{0}(\Sigma))$-fiber bundles.
\end{enumerate}
\end{definition}

In particular, the restriction of $\phi$ to any fiber is a biholomorphism isotopic to the identity.  This notion makes sense exactly because we have required the fiber bundle to have fiber $\Sigma$ with structure group $\textnormal{Diff}_{0}(\Sigma).$  

Let $X$ be a $\Sigma$-marked Riemann surface.  A small deformation of $X$ is a germ of a $\Sigma$-marked family $\mathcal{X}\rightarrow (\mathcal{B},b)$ where $b\in \mathcal{B}$ and a morphism

\begin{center}
\begin{tikzcd}
X \arrow{r}{\phi} \arrow{d}{p_{x}}
& \mathcal{X} \arrow{d}{p} \\
\{x\} \arrow{r}{f}
& \mathcal{B},
\end{tikzcd}
\end{center}
such that $f(x)=b.$  Note that it follows automatically from the definitions that $\phi: X\rightarrow p^{-1}(b)$ is a biholomorphism whose underlying smooth map is isotopic to the identity.  

A small deformation $(\mathcal{X}, \mathcal{B}, b , p)$ of $X$ is universal if for every small deformation $(\mathcal{X}_{0}, \mathcal{B}_{0}, \mathrm{B}_{0}, p_{0})$ of $X,$ there is a unique germ of a morphism 
\begin{align}
F: (\mathcal{X}_{0}, \mathcal{B}_{0}, \mathrm{B}_{0}, p_{0})\rightarrow (\mathcal{X}, \mathcal{B}, b , p),
\end{align}
such that the following diagram commutes
\begin{center}
\begin{tikzcd}
(X, \{x\}, x, p_{x}) \arrow{d}{\textnormal{id}} \arrow{r}
& (\mathcal{X}_{0}, \mathcal{B}_{0}, \mathrm{B}_{0}, p_{0}) \arrow{d}{F} \\
(X, \{x\}, x, p_{x}) \arrow{r}
& (\mathcal{X}, \mathcal{B}, b , p).
\end{tikzcd}
\end{center}

The following theorem is a summary of the main results of \cite{AC09}.

\begin{theorem}
Let $X$ be a $\Sigma$-marked Riemann surface.  Then there is an open set $U_{X}\subset \textnormal{H}^{1}(X, \Theta)\simeq \mathbb{C}^{3g-3}$ containing the origin and a universal $\Sigma$-marked small deformation of $X$ over $U.$

Moreover, the open sets $U_{X}\subset \textnormal{H}^{1}(X, \Theta)$ yield an atlas of holomorphic charts providing $\mathcal{T}_{\Sigma}$ with the structure of a Hausd\"orff complex manifold of dimension $3g-3.$

Finally, the $\Sigma$-marked universal families over $U_{X}$ glue to give a holomorphic fiber bundle $\mathcal{C}_{\Sigma}\rightarrow \mathcal{T}_{\Sigma}$ called the universal Teichm\"{u}ller curve.
\end{theorem}

Now, let $\mathcal{X}\xrightarrow{p} \mathcal{B}$ be a $\Sigma$-marked family of Riemann surfaces.  Let $\widehat{E_{\mathrm{G}}}\xrightarrow{\pi} \mathcal{X}$ be a holomorphic principal $\mathrm{G}$-bundle on $\mathcal{X}.$  In order to define families of $\Sigma$-marked $\mathrm{G}$-opers, we must develop the notion of a relative holomorphic connection on $\widehat{E_{\mathrm{G}}}.$  

To start, let $M$ and $N$ be complex manifolds and $f: M\rightarrow N$ a holomorphic map which is a fiber bundle in the $C^{\infty}$-category (in particular $f$ is a submersion).  If $\Omega_{M}^{1}$ is the sheaf of holomorphic $1$-forms on $M$, then define the sheaf of relative holomorphic $1$-forms $\Omega_{M/N}^{1}$ via the exact sequence
\begin{align}
0\rightarrow f^{\star}\Omega_{N}^{1}\rightarrow \Omega_{M}^{1}\rightarrow \Omega_{M/N}^{1}\rightarrow 0.
\end{align}

Now, consider the following commutative diagram with exact rows and columns:
\begin{center}
\begin{tikzcd}
&
0\arrow{d} 
& 0 \arrow{d}
& \\
0 \arrow{r} 
& (p\circ \pi)^{\star}\Omega_{\mathcal{B}}^{1} \arrow{r}\arrow{d}
& \pi^{\star}\Omega_{\mathcal{X}}^{1} \arrow{r}\arrow{d}
&\pi^{\star}\Omega_{\mathcal{X}/\mathcal{B}}^{1} \arrow{r}
& 0  \\

&\Omega_{\widehat{E_{\mathrm{G}}}}^{1} \arrow{r}{\simeq} \arrow{d}
& \Omega_{\widehat{E_{\mathrm{G}}}}^{1}  \arrow{d}
&
&  \\
& \Omega_{\widehat{E_{\mathrm{G}}}/\mathcal{B}}^{1} \arrow{r}{u} \arrow{d}
& \Omega_{\widehat{E_{\mathrm{G}}}/\mathcal{X}}^{1} \arrow{d} \arrow{r}
& 0 \\
& 0
& 0
& .

\end{tikzcd}
\end{center}
Chasing this diagram, we obtain an exact sequence
\begin{align}\label{rel cotangent sequence}
0\rightarrow \pi^{\star} \Omega_{\mathcal{X}/\mathcal{B}}^{1}\rightarrow \Omega_{\widehat{E_{\mathrm{G}}}/\mathcal{B}}^{1}
\rightarrow \Omega_{\widehat{E_{\mathrm{G}}}/\mathcal{X}}^{1}\rightarrow 0.
\end{align}
Dualizing yields the exact sequence
\begin{align}\label{rel tangent sequence}
0\rightarrow \Theta_{\widehat{E_{\mathrm{G}}}/\mathcal{X}}\rightarrow \Theta_{\widehat{E_{\mathrm{G}}}/\mathcal{B}}\rightarrow \pi^{\star}\Theta_{\mathcal{X}/\mathcal{B}}\rightarrow 0.
\end{align}
This is an exact sequence of $\mathrm{G}$-equivariant locally free sheaves (equivalently $\mathrm{G}$-equivariant holomorphic vector bundles) on $\widehat{E_{\mathrm{G}}}.$ 

Now comes the definition of a relative holomorphic connection on $\widehat{E_{\mathrm{G}}}.$  
\begin{definition}
A relative holomorphic connection on $\widehat{E_{\mathrm{G}}}$ is a $\mathrm{G}$-equivariant splitting of the exact sequence \eqref{rel tangent sequence}
.\end{definition}

\textbf{Remark:}  If $\mathcal{B}=\{\textnormal{pt}\},$ then \eqref{rel tangent sequence} is the usual tangent sequence associated to a principal $\mathrm{G}$-bundle on $X,$ and this is the usual sheaf-theoretic definition of a holomorphic connection (see Section  \ref{A bundles}).

Abusing notation, we denote by $\Theta_{\widehat{E_{\mathrm{G}}}/\mathcal{X}}$ the holomorphic vector bundle associated to the locally free sheaf $\Theta_{\widehat{E_{\mathrm{G}}}/\mathcal{X}}$ on $\widehat{E_{\mathrm{G}}}.$

Given $g\in G,$ we denote the right holomorphic $\mathrm{G}$-action by
$R_{g}: \widehat{E_{\mathrm{G}}}\rightarrow \widehat{E_{\mathrm{G}}}.$  Note that $\Theta_{\widehat{E_{\mathrm{G}}}/\mathcal{B}}$ is a $\mathrm{G}$-equivariant sub-sheaf of $\Theta_{\widehat{E_{\mathrm{G}}}},$ and therefore the infinitesimal $\mathrm{G}$-action may be restricted to sections of $\Theta_{\widehat{E_{\mathrm{G}}}/\mathcal{B}}.$

Perhaps a more familiar (equivalent) definition of a relative holomorphic connection on $\widehat{E_{\mathrm{G}}}$ is the following.

\begin{definition}
A relative holomorphic connection on $\widehat{E_{\mathrm{G}}}$ is a holomorphic map
\begin{align}
\overline{\omega}: \Theta_{\widehat{E_{\mathrm{G}}}/\mathcal{B}} \rightarrow \mathfrak{g},
\end{align}
which is linear in the fibers and which satisfies:
\begin{enumerate}
\item $R_{g}^{\star}\overline{\omega}=\textnormal{Ad}(g^{-1})\circ\overline{\omega}$
for all $g\in G.$
\item If $X\in \mathfrak{g}$ and $X^{\sharp}$ is the $\mathrm{G}$-invariant section of $\Theta_{\widehat{E_{\mathrm{G}}}/\mathcal{X}}$ given by the infinitesimal $\mathrm{G}$-action on $\widehat{E_{\mathrm{G}}},$ then $\overline{\omega}(X^{\sharp})=X.$
\end{enumerate}
\end{definition}

As with ordinary connections on principal $\mathrm{G}$-bundles, relative holomorphic connections are an affine space
modelled on the vector space of holomorphic sections $\textnormal{H}^{0}(\mathcal{X}, \Omega_{\mathcal{X}/\mathcal{B}}^{1}\otimes \widehat{E_{\mathrm{G}}}[\mathfrak{g}]).$

If 
\begin{align}
d_{(\pi,p)}: \Omega_{\widehat{E_{\mathrm{G}}}/\mathcal{B}}^{1}\rightarrow \Omega_{\widehat{E_{\mathrm{G}}}/\mathcal{B}}^{2}
\end{align}
is the relative exterior derivative along the fibers of $p\circ \pi$, then the curvature of $\overline{\omega}$ is defined by
\begin{align}
F(\overline{\omega})=d_{(\pi, p)}\overline{\omega}+\frac{1}{2}[\overline{\omega}, \overline{\omega}],
\end{align}
and descends to a holomorphic section $F(\overline{\omega})\in \textnormal{H}^{0}(\mathcal{X}, \Omega_{\mathcal{X}/\mathcal{B}}^{2}\otimes \widehat{E_{\mathrm{G}}}[\mathfrak{g}]).$  Since $\mathcal{X}$ is a family of Riemann surfaces, $\Omega_{\mathcal{X}/\mathcal{B}}^{2}$ is the zero sheaf and such a section is necessarily zero.  Therefore, every holomorphic relative connection is automatically flat.

Finally, if $\mathrm{H}<\mathrm{G}$ is a closed complex subgroup and $\widehat{E_{\mathrm{H}}}$ is a holomorphic reduction of $\widehat{E_{\mathrm{G}}}$ to $\mathrm{H},$ then the relative second fundamental form $\overline{\Psi}$ of $\overline{\omega}$ relative to $\widehat{E_{\mathrm{H}}}$ is defined as the composition
\begin{align}
\Theta_{\widehat{E_{\mathrm{H}}}/\mathcal{B}}\rightarrow \Theta_{\widehat{E_{\mathrm{G}}}/\mathcal{B}} \xrightarrow{\overline{\omega}} \mathfrak{g}\rightarrow \mathfrak{g}/\mathfrak{h}
\end{align}
and descends to a holomorphic section $\overline{\Psi}\in \textnormal{H}^{0}(\mathcal{X}, \Omega_{\mathcal{X}/\mathcal{B}}^{1}\otimes \widehat{E_{\mathrm{H}}}(\mathfrak{g}/\mathfrak{h})).$

We now arrive at the definition of a $\Sigma$-marked family of $\mathrm{G}$-opers.

\begin{definition}\label{family g oper}
Let $\mathrm{G}$ be a connected complex semi-simple Lie group with a fixed Borel subgroup $\mathrm{B}<\mathrm{G}.$  

A $\Sigma$-marked family of $\mathrm{G}$-opers consists of a tuple $(\widehat{E_{\mathrm{G}}}, \widehat{E_{\mathrm{B}}}, \mathcal{X}, \mathcal{B}, \overline{\omega})$ such that
\begin{enumerate}
\item The pair $(\mathcal{X}, \mathcal{B})$ is a $\Sigma$-marked family of Riemann surfaces.  
\item $\widehat{E_{\mathrm{G}}}$ is a holomorphic, right principal $\mathrm{G}$-bundle over $\mathcal{X}$ and $\widehat{E_{\mathrm{B}}}$ is a holomorphic reduction of structure to the Borel subgroup $\mathrm{B}<\mathrm{G}.$  
\item $\overline{\omega}$ is a relative holomorphic connection on $\widehat{E_{\mathrm{G}}}.$
\item For all non-zero vectors $v\in \Theta_{\mathcal{X}/\mathcal{B}},$
\begin{align}
\overline{\Psi}(v)\in \widehat{E_{\mathrm{B}}}[\mathcal{O}]
\end{align}
where $\mathcal{O}\subset \mathfrak{g}^{-1}/\mathfrak{b}$ is the unique open $\mathrm{B}$-orbit.
\end{enumerate}
\end{definition}

\textbf{Remark}: If $(\mathcal{X}, \mathcal{B})=(X, \{x\})$ is a trivial family over a point $\{x\}$, then a $\Sigma$-marked family of $\mathrm{G}$-opers is identical to a $\Sigma$-marked $\mathrm{G}$-oper.  In general, this definition formalizes the notion of a $\Sigma$-marked family of Riemann surfaces where each fiber in the family is equipped with a $\mathrm{G}$-oper structure.

A small deformation of a $\Sigma$-marked $\mathrm{G}$-oper $(E_{\mathrm{G}}, E_{\mathrm{B}}, \omega, X)$ is defined in the obvious way, and such a deformation is said to be universal if any other deformation is pulled back from this one by a unique morphism.  Since all of these definitions are obvious adaptations of the definitions given for $\Sigma$-marked families of Riemann surfaces, we elect to not explicitly spell out the enhanced definitions here.

Finally, we can state the main theorem whose hypotheses are contained in this section.  The proof of this theorem will occupy the next few sub-sections.

\begin{theorem}\label{uni def}
Let $\mathrm{G}$ be a complex simple Lie group of adjoint type.  Then, given any $\Sigma$-marked $\mathrm{G}$-oper $(E_{\mathrm{G}}, E_{\mathrm{B}}, \omega, X)$, there is a universal small deformation whose base $\mathcal{B}$ is of complex dimension $\textnormal{dim}_{\mathbb{C}}(G)(g-1)+(3g-3)$ where $g$ is the genus of $\Sigma.$  
\end{theorem}

\subsection{Atiyah bundles and flat connections}\label{A bundles}

In order to study the infinitesimal deformation theory of $\Sigma$-marked $\mathrm{G}$-opers, we now review the sheaf-theoretic way of thinking about connections on principal bundles. 

Let $\mathrm{G}$ be a connected complex semi-simple Lie group and $E_{\mathrm{G}}$ a holomorphic principal $\mathrm{G}$-bundle over a $\Sigma$-marked Riemann surface $X.$  Consider the sub-sheaf $\mathcal{G}\subset\Theta_{E_{\mathrm{G}}}$ of the tangent sheaf of $E_{\mathrm{G}}$ whose local sections consist of $\mathrm{G}$-invariant holomorphic vector fields on $E_{\mathrm{G}}.$  

 By a theorem of Atiyah \cite{ATI57}, the sheaf $\mathcal{G}$ descends to a locally free sheaf on $X$: namely there exists a unique locally-free sheaf $\mathcal{A}(E_{\mathrm{G}})$ on $X$, called the \emph{Atiyah} sheaf (equivalently bundle), such that the pullback $\pi^{\star}\mathcal{A}(E_{\mathrm{G}})$ of $\mathcal{A}(E_{\mathrm{G}})$ to $E_{\mathrm{G}}$ via the projection map $E_{\mathrm{G}}\xrightarrow{\pi} X$ is equal to $\mathcal{G}.$  Note that Beilinson-Drinfeld \cite{BD05} call this the (Atiyah)-algebroid of infinitesimal symmetries of $E_{\mathrm{G}}.$ 

There is an exact sequence of locally-free sheaves (equivalently holomorphic vector bundles) on $X$
\begin{align}\label{atiyah sequence}
0\rightarrow \mathcal{E}_{\mathrm{G}}[\mathfrak{g}]\rightarrow \mathcal{A}(E_{\mathrm{G}})\xrightarrow{\sigma_{\mathrm{G}}} \Theta\rightarrow 0,
\end{align}
where $\mathcal{E}_{\mathrm{G}}[\mathfrak{g}]$ is the sheaf of vertical $\mathrm{G}$-invariant vector fields on $E_{\mathrm{G}}.$  This is the sheaf of sections of the holomorphic vector bundle $E_{\mathrm{G}}[\mathfrak{g}].$ 
The map $\sigma_{\mathrm{G}}:  \mathcal{A}(E_{\mathrm{G}})\rightarrow \Theta$ is called the \emph{symbol} map.

\textbf{Example:}  If $V$ is a rank $n$-holomorphic vector bundle over $X,$ let $\mathcal{D}_{0}^{1}(V)$ denote the sheaf of first-order differential operators on holomorphic sections of $V$ which have scalar principal symbol.  

Given a local holomorphic frame $\{e_{i}\}_{i=1}^{n}$, a local section $P$ of the sheaf $\mathcal{D}_{0}^{1}(V)$
is given by an expression of the form
\begin{align}
P=\sum_{i=1}^{n} \xi\otimes \textnormal{Id} + B,
\end{align}
where $\xi$ is a local section of $\Theta$ and $\mathrm{B}$ is a local section of $\textnormal{End}(V).$  

The action of $P$ on a section $s=\sum_{i=1}^{n} s^{i}\otimes e_{i}$ in this trivialization is given by
\begin{align}
P(s)=\sum_{i=1}^{n} \xi(s^{i})\otimes e_{i}+s^{i}\otimes B(e_{i}).
\end{align}

If $E_{\textnormal{GL}_{n}(\mathbb{C})}$ is the holomorphic principal $\textnormal{GL}_{n}( \mathbb{C})$-bundle associated to $V,$ then the Atiyah sequence \eqref{atiyah sequence} for $E_{\textnormal{GL}_{n}(\mathbb{C})}$ is equivalent to the exact sequence
\begin{align}\label{SES vb}
0\rightarrow \textnormal{End}(V)\rightarrow \mathcal{D}_{0}^{1}(V)\xrightarrow{\sigma_{V}} \Theta\rightarrow 0
\end{align}
where $\sigma_{V}(P)=\xi$ is the \emph{principal symbol} of the differential operator $P.$  This example explains the terminology \emph{symbol map} for the map $\mathcal{A}(E_{\mathrm{G}})\xrightarrow{\sigma_{\mathrm{G}}} \Theta.$  

Finally, if $\nabla$ is a holomorphic connection on $V,$ then for any local section $\xi$ of $\Theta,$ the operator $\nabla_{\xi}$ is a local section of $\mathcal{D}_{0}^{1}(V).$  Moreover, the fact that $\nabla$ satisfies the Leibniz rule implies that $\nabla$ defines a holomorphic splitting of \eqref{SES vb}.  This shows that holomorphic splittings of \eqref{SES vb} are equivalent to holomorphic connections on $V.$

Generally, in light of the fact that a holomorphic connection on $E_{\mathrm{G}}$ is equivalent to a $\mathrm{G}$-equivariant horizontal, holomorphic splitting $TE_{\mathrm{G}}\simeq \mathcal{H}\oplus \mathcal{V}$ (e.g. a $\mathrm{G}$-equivariant horizontal, holomorphic distribution), we arrive at Atiyah's \cite{ATI57} definition of a holomorphic connection on $E_{\mathrm{G}}.$

\begin{definition}
A holomorphic connection on $E_{\mathrm{G}}$ is a holomorphic splitting of the symbol map
\begin{align}
\mathcal{A}(E_{\mathrm{G}})\xrightarrow{\sigma_{\mathrm{G}}} \Theta.
\end{align}
\end{definition}

We can recast the definition of a $\mathrm{G}$-oper in this language.  Namely, the locally-free sub-sheaf $\mathcal{E}_{\mathrm{B}}[\mathfrak{g}^{-1}]\subset \mathcal{E}_{\mathrm{G}}[\mathfrak{g}]$ determines a locally-free sub-sheaf $\mathcal{A}^{-1}\subset \mathcal{A}(E_{\mathrm{G}})$ and a short exact sequence
\begin{align}\label{ses atiyah}
0\rightarrow \mathcal{E}_{\mathrm{B}}[\mathfrak{g}^{-1}]\rightarrow \mathcal{A}^{-1}\xrightarrow{\sigma_{\mathrm{G}, -1}} \Theta\rightarrow 0.
\end{align}
Then, a $\mathrm{G}$-oper structure is a holomorphic splitting $\omega$ of \eqref{ses atiyah} such that
the composition
\begin{align}
\Theta\xrightarrow{\omega} \mathcal{A}^{-1}\rightarrow \mathcal{A}^{-1}/\mathcal{A}(E_{\mathrm{B}})\simeq \mathcal{E}_{\mathrm{B}}[\mathfrak{g}^{-1}/\mathfrak{b}]
\end{align}
lies in the unique open $\mathrm{B}$-orbit $\mathcal{O}\subset\mathfrak{g}^{-1}/\mathfrak{b}$ for every non-zero vector in $\Theta.$

\subsection{Deformation theory of $\mathrm{G}$-opers}

In this section we study the infinitesimal deformation theory of $\mathrm{G}$-opers.  Our efforts in the previous Section \ref{A bundles} will come to fruition here as the deformation theory is nicely captured using the theory of Atiyah bundles.

Before diving in, we need some definitions.  In Section \ref{families} we introduced the notion of families of $\mathrm{G}$-opers over complex manifolds.  Unfortunately, this is not the right context for the study of infinitesimal deformation theory.  We remedy this here.  In this subsection, we assume some familiarity with complex analytic spaces.  The interested reader may find all of the relevant background material in the book \cite{GLS07}.

Let $\mathbb{D}^{(n)}$ denote the complex analytic space associated to the holomorphic function
\begin{align}
f: \mathbb{C}&\rightarrow \mathbb{C} \\
           z&\mapsto z^{n+1}.
\end{align}
Any complex analytic space with underlying topolgical space a single point is called a \emph{fat point}.  As a locally ringed space,
\begin{align}
(\mathbb{D}^{(n)}, \mathcal{O}_{\mathbb{D}^{(n)}})\simeq (\{0\}, \mathbb{C}[\varepsilon]/(\varepsilon^{n+1})).
\end{align}
An $n$-th order deformation of a $\Sigma$-marked Riemann surface $X$ consists of a commutative diagram
\begin{center}
\begin{tikzcd}
X \arrow{r} \arrow{d}
& \mathcal{X} \arrow{d}{p} \\
\mathbb{D}^{(0)} \arrow{r}
& \mathbb{D}^{(n)}.
\end{tikzcd}
\end{center}
Here, all maps are maps of complex analytic spaces, and $p$ is assumed to be a flat map of complex analytic spaces.  Furthermore, the top horizontal arrow is a closed embedding of complex analytic spaces.

Following the rubric of section \ref{families} and accepting some familiarity with the theory of complex analytic spaces, there is a straightforward notion of a $\Sigma$-marked $\mathrm{G}$-oper over the family $\mathcal{X}\xrightarrow{p} \mathbb{D}^{(n)},$ and an $n$-th order deformation of a $\Sigma$-marked $\mathrm{G}$-oper $(E_{\mathrm{G}}, E_{\mathrm{B}}, \omega, X)$.  With these definitions in place,  the following theorem solves the problem of infinitesimal deformations of $\mathrm{G}$-opers.

\begin{theorem}\label{kuranishi family}
Let $\Xi:=(E_{\mathrm{G}}, E_{\mathrm{B}}, \omega, X)$ be a $\Sigma$-marked $\mathrm{G}$-oper.  There is a two term complex
of locally-free sheaves
\begin{align}\label{2 term}
\mathcal{A}^{\bullet}:\mathcal{A}(E_{\mathrm{B}})\xrightarrow{[\hat{\omega}, \ ]} \mathcal{K}\otimes \mathcal{E}_{\mathrm{B}}[\mathfrak{g}^{-1}]
\end{align}
such that,
\begin{enumerate}
\item The $0$-th hyper-cohomology $\mathbb{H}^{0}(X, \mathcal{A}^{\bullet})$ is in bijection with infinitesimal automorphisms of the $\mathrm{G}$-oper $\Xi.$
\item The $1$-st hyper-cohomology $\mathbb{H}^{1}(X, \mathcal{A}^{\bullet})$ is in bijection with isomorphism classes of first-order deformations of the $\mathrm{G}$-oper $\Xi$ such that $0\in \mathbb{H}^{1}(X, \mathcal{A}^{\bullet})$ corresponds to the trivial first order deformation.
\item There is a quadratic obstruction map
\begin{align}
\textnormal{Ob}:\mathbb{H}^{1}(X, \mathcal{A}^{\bullet})\rightarrow \mathbb{H}^{2}(X, \mathcal{A}^{\bullet})
\end{align}
such that a first-order deformation in $\mathbb{H}^{1}(X, \mathcal{A}^{\bullet})$ extends to a second-order deformation if and only if its image under the map $\textnormal{Ob}$ vanishes.
\end{enumerate}
Finally, we have the following equalities
\begin{enumerate}
\item $\mathbb{H}^{0}(X, \mathcal{A}^{\bullet})=\{0\}.$
\item If $\mathrm{G}$ is complex simple of adjoint type, then $\textnormal{dim}_{\mathbb{C}}(\mathbb{H}^{1}(X, \mathcal{A}^{\bullet}))=(g-1)\textnormal{dim}_{\mathbb{C}}(G)+(3g-3)$ where $\mathrm{G}$ is the genus of $\Sigma.$  
\item If $\mathrm{G}$ is complex simple of adjoint type, then $\mathbb{H}^{2}(X, \mathcal{A}^{\bullet})=\{0\}.$
\end{enumerate}
\end{theorem}

\begin{proof}
Let $(E_{\mathrm{G}}, E_{\mathrm{B}}, \omega, X)$ be a $\Sigma$-marked $\mathrm{G}$-oper.  Recall that $\mathcal{A}^{-1}\subset \mathcal{A}(E_{\mathrm{G}})$ is the locally free sub-sheaf such that 
the holomorphic flat connection $\omega$ is a holomorphic splitting of
\begin{align}\label{oper splitting}
0\rightarrow \mathcal{E}_{\mathrm{B}}[\mathfrak{g}^{-1}]\rightarrow \mathcal{A}^{-1}\xrightarrow{\sigma_{\mathrm{G},-1}} \Theta\rightarrow 0.
\end{align}
Therefore, the holomorphic flat connection $\omega$ may be viewed as a global holomorphic section $\hat{\omega}\in \textnormal{H}^{0}(X, \mathcal{K}\otimes \mathcal{A}^{-1}).$  

Locally on $X,$ we may write $\hat{\omega}=\sum_{i} \beta_{i}\otimes u_{i}$ where $\beta_{i}$ are locally defined holomorphic $1$-forms on $X$ and the $u_{i}$ are local sections of $\mathcal{A}^{-1}.$ If $s$ is a local section of $\mathcal{A}(E_{\mathrm{B}}),$
then define
\begin{align}\label{bracket}
[\hat{\omega}, s]:=\sum_{i} \left(\beta_{i}\otimes [u_{i}, s] -\mathcal{L}_{\sigma_{\mathrm{B}}(s)}\beta_{i}\otimes u_{i}\right)
\end{align}
where $\mathcal{L}_{\sigma_{\mathrm{B}}(s)}\beta_{i}$ is the Lie derivative of the local holomorphic $1$-form $\beta_{i}$ along the local holomorphic vector field $\sigma_{\mathrm{B}}(s).$  

The bracket \eqref{bracket} defines a sheaf map
\begin{align}\label{complex first}
\mathcal{A}(E_{\mathrm{B}})\xrightarrow{[\hat{\omega}, \ ]} \mathcal{K}\otimes \mathcal{A}^{-1}.
\end{align}
Since $\hat{\omega}$ is a splitting of \eqref{oper splitting},
\begin{align}
\sum_{i} \alpha_{i}\otimes \sigma_{\mathrm{G}, -1}(u_{i})=1
\end{align}
as a local section of $\mathcal{K}\otimes \Theta\simeq \mathcal{O}.$ Therefore
\begin{align}
0&=\mathcal{L}_{\sigma_{\mathrm{B}}(s)}\left(\sum_{i} \alpha_{i}\otimes \sigma_{\mathrm{G}, -1}(u_{i})\right) \\
&=\sum_{i} \left(\mathcal{L}_{\sigma_{\mathrm{B}}(s)}\alpha_{i}\otimes \sigma_{\mathrm{G}, -1}(u_{i})+\alpha_{i}\otimes [\sigma_{\mathrm{B}}(s), \sigma_{\mathrm{G}, -1}(u_{i})]\right).
\end{align}
Given a local section $s$ of $\mathcal{A}(E_{\mathrm{B}})$, we compute
\begin{align}
\textnormal{id}\otimes \sigma_{\mathrm{G}, -1}([\hat{\omega}, s])&=\sum_{i} \left(\alpha_{i}\otimes \sigma_{\mathrm{G}, -1}[u_{i},s]-\mathcal{L}_{\sigma_{\mathrm{B}}(s)}\alpha_{i}\otimes \sigma_{\mathrm{G},-1}(u_{i})\right) \\
&= \sum_{i}\left( \alpha_{i}\otimes \sigma_{\mathrm{G},-1}[u_{i},s]+\alpha_{i}\otimes [\sigma_{\mathrm{B}}(s), \sigma_{\mathrm{G},-1}(u_{i})]\right) \\
&=0,
\end{align}
where we have used the fact that the symbol map is a map of sheaves of Lie algebras which is functorial with respect to sub-sheaves of $\mathcal{A}(E_{\mathrm{G}}).$  

Since $\textnormal{ker}(\sigma_{\mathrm{G}, -1})=\mathcal{E}_{\mathrm{B}}[\mathfrak{g}^{-1}],$ the map \eqref{complex first} lifts to
\begin{align}
\mathcal{A}^{\bullet}:=\mathcal{A}(E_{\mathrm{B}})\xrightarrow{[\hat{\omega}, \ ]} \mathcal{K}\otimes \mathcal{E}_{\mathrm{B}}[\mathfrak{g}^{-1}].
\end{align}
Proposition $4.4$ in \cite{CHE12} implies that the complex $\mathcal{A}^{\bullet}$ governs the deformation theory of the $\Sigma$-marked $\mathrm{G}$-oper $(E_{\mathrm{G}}, E_{\mathrm{B}}, \omega, X).$  This means that:
\begin{enumerate}
\item  Infinitesimal automorphisms are in bijection with $\mathbb{H}^{0}(X, \mathcal{A}^{\bullet}).$
\item  First order deformations up to isomorphism are in bijection with $\mathbb{H}^{1}(X, \mathcal{A}^{\bullet}).$
\item  Obstructions to lifting a first order deformation to a second order deformation lie in $\mathbb{H}^{2}(X, \mathcal{A}^{\bullet}).$
            \end{enumerate}
            
\textbf{Remark}: Specifically, what Chen proves in \cite{CHE12} is that given the data $(E_{\mathrm{G}}, E_{\mathrm{B}}, \omega, X),$ the complex $\mathcal{A}^{\bullet}$ controls the deformation theory of this tuple, where we constrain the second fundamental form of $\omega$ relative to $E_{\mathrm{B}}$ to lie in $E_{\mathrm{B}}[\mathfrak{g}^{-1}/\mathfrak{b}].$  But, since lying in the open orbit is a open condition, any small deformation satisfying this property automatically yields a deformation of the $\Sigma$-marked $\mathrm{G}$-oper $(E_{\mathrm{G}}, E_{\mathrm{B}}, \omega, X).$  This proves the first part the Theorem \ref{kuranishi family}.

Now consider the short exact sequence of complexes 
\begin{align}\label{SES X to S}
0\rightarrow \mathcal{A}_{0}^{\bullet}\rightarrow \mathcal{A}^{\bullet}\rightarrow \Theta^{0}\rightarrow 0
\end{align}
defined by
\begin{center}
\begin{tikzcd}
0 \arrow{d} 
& 0 \arrow{d} \\
E_{\mathrm{B}}[\mathfrak{b}] \arrow{r}{[\hat{\omega}, \ ]} \arrow{d}
& \mathcal{K}\otimes \mathcal{E}_{\mathrm{B}}[\mathfrak{g}^{-1}] \arrow{d} \\
\mathcal{A}(E_{\mathrm{B}}) \arrow{r}{[\hat{\omega}, \ ]} \arrow{d}{\sigma_{\mathrm{B}}}
&  \mathcal{K}\otimes \mathcal{E}_{\mathrm{B}}[\mathfrak{g}^{-1}] \arrow{d}\\
\Theta \arrow{r} \arrow{d} 
&0 \\
0.
\end{tikzcd}
\end{center}
Since $\textnormal{H}^{0}(X, \Theta)=\{0\}$, the long exact sequence in hyper-cohomology implies that
$\mathbb{H}^{0}(X, \mathcal{A}_{0}^{\bullet})\simeq \mathbb{H}^{0}(X, \mathcal{A}^{\bullet}).$  But, the automorphism group of a $\mathrm{G}$-oper on $X$ is finite \cite{BD05} (equal to the center of $\mathrm{G}$), so $\mathbb{H}^{0}(X, \mathcal{A}_{0}^{\bullet})=\{0\}.$ Hence $\mathbb{H}^{0}(X, \mathcal{A}^{\bullet})=\{0\}.$

Now, we turn to the final statement of Theorem \ref{kuranishi family}.  By Theorem \ref{param opers 2} and using the fact that $\mathrm{G}$ is adjoint simple, there is a parameterization of $\Sigma$-marked $\mathrm{G}$-opers on $X$ by the vector space
\begin{align}
\bigoplus_{i=1}^{\ell}\textnormal{H}^{0}(X, \mathcal{K}^{m_{i}+1}).
\end{align}
Appealing again to \cite{CHE12}, the complex $\mathcal{A}_{0}$ governs the deformation theory of $\Sigma$-marked $\mathrm{G}$-opers on the fixed $\Sigma$-marked Riemann surface $X,$ and therefore 
\begin{align}\label{dim count}
\textnormal{dim}_{\mathbb{C}}(\mathbb{H}^{1}(X, \mathcal{A}_{0}^{\bullet}))&=\textnormal{dim}_{\mathbb{C}}\left(\bigoplus_{i=1}^{\ell}\textnormal{H}^{0}(X, \mathcal{K}^{m_{i}+1})\right) \\
&=\textnormal{dim}_{\mathbb{C}}(G)(g-1),
\end{align}
where the last equality follows by an application of the Riemann-Roch formula.

We first show that $\mathbb{H}^{2}(X, \mathcal{A}_{0}^{\bullet})=\{0\}.$  This will be achieved by examining the explicit models from section \ref{models} and an application of the Grothiendieck-Riemann-Roch theorem.

Using the explicit models of section \ref{models}, there are $C^{\infty}$-isomorphisms
\begin{align}
\mathcal{E}_{\mathrm{B}}[\mathfrak{b}]&\simeq \bigoplus_{i=0}^{m_{\ell}} \mathcal{K}^{i}\otimes \mathfrak{g}_{i} \\
\mathcal{K}\otimes \mathcal{E}_{\mathrm{B}}[\mathfrak{g}^{-1}]&\simeq \bigoplus_{i=-1}^{m_{\ell}} \mathcal{K}^{i+1}\otimes \mathfrak{g}_{i}.
\end{align}
With this informaton, the relevant ranks and degrees of the bundles in question are:
\begin{align}
\textnormal{deg}(\mathcal{E}_{\mathrm{B}}[\mathfrak{b}])&=\sum_{i=0}^{m_{\ell}}
i(2g-2)\textnormal{dim}_{\mathbb{C}}(\mathfrak{g}_{i}). \\
\textnormal{rank}(\mathcal{E}_{\mathrm{B}}[\mathfrak{b}])&=\textnormal{dim}_{\mathbb{C}}(\mathfrak{b}).\\
\textnormal{deg}(\mathcal{K}\otimes\mathcal{E}_{\mathrm{B}}[\mathfrak{g}^{-1}])&=\sum_{i=-1}^{m_{\ell}}
(i+1)(2g-2)\textnormal{dim}_{\mathbb{C}}(\mathfrak{g}_{i}). \\
\textnormal{rank}(\mathcal{K}\otimes \mathcal{E}_{\mathrm{B}}[\mathfrak{g}^{-1}])&=\textnormal{dim}_{\mathbb{C}}(\mathfrak{g}^{-1}).
\end{align}
Therefore, using \eqref{dim count} and the fact that $\mathbb{H}^{0}(X, \mathcal{A}_{0}^{\bullet})=\{0\}$,  an application of the Grothiendick-Riemann-Roch theorem reveals:
\begin{align}
(1-g)\textnormal{dim}_{\mathbb{C}}(G)+\textnormal{dim}_{\mathbb{C}}(\mathbb{H}^{2}(X, \mathcal{A}_{0}^{\bullet}))&=(1-g)\textnormal{rk}(\mathcal{A}_{0}^{\bullet})+\textnormal{deg}(\mathcal{A}_{0}^{\bullet})
\\
&=(1-g)\left(\textnormal{rk}(\mathcal{E}_{\mathrm{B}}[\mathfrak{b}])-\textnormal{rk}(\mathcal{K}\otimes \mathcal{E}_{\mathrm{B}}[\mathfrak{g}^{-1}])\right)  \\
&+ \left(\textnormal{deg}(\mathcal{E}_{\mathrm{B}}[\mathfrak{b}])-\textnormal{deg}(\mathcal{K}\otimes \mathcal{E}_{\mathrm{B}}[\mathfrak{g}^{-1}])\right) \\
&=(g-1)\ell -(2g-2)\textnormal{dim}_{\mathbb{C}}(\mathfrak{b}) \\
&=(1-g)\left(2\textnormal{dim}_{\mathbb{C}}(\mathfrak{b})-\ell\right) \\
&=(1-g)\textnormal{dim}_{\mathbb{C}}(G).
\end{align}
Hence, $\mathbb{H}^{2}(X, \mathcal{A}_{0}^{\bullet})=\{0\}.$  

Returning to the short exact sequence of complexes \eqref{SES X to S}
\begin{align}
0\rightarrow \mathcal{A}_{0}^{\bullet}\rightarrow \mathcal{A}^{\bullet}\rightarrow \Theta\rightarrow 0,
\end{align}
  and using that $\textnormal{H}^{0}(X, \Theta)=H^{2}(X, \Theta)=\{0\},$ the long exact sequence derived from \eqref{SES X to S} yields the exact sequence:
\begin{align}
0\rightarrow \mathbb{H}^{1}(X, \mathcal{A}_{0}^{\bullet})\rightarrow \mathbb{H}^{1}(X, \mathcal{A}^{\bullet})\rightarrow \textnormal{H}^{1}(X, \Theta)\rightarrow 0 \rightarrow \mathbb{H}^{2}(X, \mathcal{A}^{\bullet})\rightarrow 0.
\end{align}
This simultaneously proves that
\begin{align}
\textnormal{dim}_{\mathbb{C}}(\mathbb{H}^{1}(X, \mathcal{A}^{\bullet}))=(g-1)\textnormal{dim}_{\mathbb{C}}(G)+(3g-3)
\end{align}
 where $g$ is the genus of $\Sigma$ and $\mathbb{H}^{2}(X, \mathcal{A}^{\bullet})=\{0\}.$  This completes the proof.
\end{proof}
The above proof also establishes the following:
\begin{corollary}\label{cor sur} 
Let $\mathrm{G}$ be a complex simple Lie group of adjoint type and $(E_{\mathrm{G}}, E_{\mathrm{B}}, \omega, X)$ be a $\Sigma$-marked $\mathrm{G}$-oper.  Then, the induced map
\begin{align}
\mathbb{H}^{1}(X, \mathcal{A}^{\bullet})\rightarrow \textnormal{H}^{1}(X, \Theta)
\end{align}
is surjective.
\end{corollary}
Now, we can apply the Kuranishi method and obtain the following theorem.  
\begin{theorem}\label{Kuranishi family}
Let $\mathrm{G}$ be a complex simple Lie group of adjoint type.  Given any $\Sigma$-marked $\mathrm{G}$-oper $\Xi=(E_{\mathrm{G}}, E_{\mathrm{B}}, \omega, X)$, there is an open set $U\subset \mathbb{H}^{1}(X, \mathcal{A}^{\bullet})$ containing the origin and a universal small deformation 
\begin{align}
\left(\widehat{E_{\mathrm{G}}}, \widehat{E_{\mathrm{B}}}, \overline{\omega}, \mathcal{X}, (U,0)\right)
\end{align}
  of $\Xi.$  
\end{theorem}

\textbf{Remark}: Since we are omitting the proof of this result, let us make some comments.  The aforementioned \emph{Kuranishi method} was introduced, building on the work of Kodaira and Spencer \cite{KOD86}, by Kuranishi \cite{KUR62} where he established that if $M$ is a compact complex manifold with $H^{2}(M, \Theta_{M})=\{0\},$ then there exists a universal small deformation of $M$ parameterized by an open set in $\textnormal{H}^{1}(M, \Theta_{M})$ containing the origin.  This was followed by an explosive development in deformation theory which continues to this day: we cite \cite{MAN04} for a recount of the basic theory and applications.  

For our purposes, the results of \cite{CHE12} establish the infinitesimal deformation theory of a $\Sigma$-marked $\mathrm{G}$-oper as discussed in the previous theorem.  The question of building a small universal deformation of a pair $(E_{\mathrm{G}}, X)$ was solved in \cite{CS16}.  

 We make two remarks here: since $E_{\mathrm{B}}$ is a reduction of structure of $E_{\mathrm{G}},$ a deformation of $E_{\mathrm{B}}$ induces a unique deformation of $E_{\mathrm{G}},$ and thus deformations of the triple $(E_{\mathrm{B}}, \omega, X)$ are equivalent to deformations of the tuple $(E_{\mathrm{G}}, E_{\mathrm{B}}, \omega, X).$  This explains why $\mathcal{A}(E_{\mathrm{G}})$ does not appear in the two-term complex \eqref{2 term}.  
 
 Also, in the paper \cite{CS16}, they only study holomorphic vector bundles. But, since $\mathrm{G}$ is of adjoint type, the deformation theory of $E_{\mathrm{G}}$ is completely equivalent to the deformation theory of the holomorphic vector bundle $E_{\mathrm{G}}[\mathfrak{g}],$ and therefore the results of \cite{CS16} concerning holomorphic vectors apply with no essential modifications. 

A combination of the techniques in the aforementioned papers \cite{CHE12}, \cite{CS16} yields a proof of Theorem \ref{Kuranishi family} where no new ideas are necessary.  Hence, we conclude the discussion of Theorem \ref{Kuranishi family} here. 

\subsection{Global structure of $\Sigma$-marked $\mathrm{G}$-opers}

In this section, we finally arrive at the proof of Theorem \ref{def space}, providing the moduli space of $\Sigma$-marked $\mathrm{G}$-opers with a canonical complex structure when $\mathrm{G}$ is a complex simple Lie group of adjoint type.

\begin{theorem}\label{global opers}
Let $\mathrm{G}$ be a complex simple Lie group of adjoint type.  The moduli space of $\Sigma$-marked $\mathrm{G}$-opers $\mathcal{O}\mathfrak{p}_{\Sigma}(\mathrm{G})$ admits the structure of a Hausd\"orff complex manifold of dimension $\textnormal{dim}_{\mathbb{C}}(G)(g-1)+(3g-3)$ where $\mathrm{G}$ is the genus of $\Sigma.$  

Moreover, the natural map
\begin{align}
  \mathcal{P}: \mathcal{O}\mathfrak{p}_{\Sigma}(\mathrm{G})\rightarrow \mathcal{T}_{\Sigma}
  \end{align}
  is a holomorphic submersion.
  
  Finally, there is a commutative diagram
  \begin{center}
  \begin{tikzcd}
  T_{[(E_{\mathrm{G}}, E_{\mathrm{B}}, \omega, X)]}\mathcal{O}\mathfrak{p}_{\Sigma}(\mathrm{G}) \arrow{r}{d\mathcal{P}} \arrow{d} 
  & T_{[X]}\mathcal{T}_{\Sigma} \arrow{d} \\
  \mathbb{H}^{1}(X, \mathcal{A}^{\bullet}) \arrow{r}
  & \textnormal{H}^{1}(X, \Theta_{X})
  \end{tikzcd}
  \end{center}
  where the lower horizontal arrow is the induced map in hyper-cohomology from the exact sequence
  \begin{align}
  0\rightarrow \mathcal{A}_{0}^{\bullet}\rightarrow \mathcal{A}^{\bullet}\rightarrow \Theta\rightarrow 0.
  \end{align}
  \end{theorem}
  
\begin{proof}
Let $\Xi:=(E_{\mathrm{G}}, E_{\mathrm{B}}, \omega, X)$ be a $\Sigma$-marked $\mathrm{G}$-oper.  Let $(\mathcal{B}_{\Xi},0_{\Xi})\subset (\mathbb{H}^{1}(X, \mathcal{A}^{\bullet}), 0)$ be the pointed base of a $\Sigma$-marked universal small deformation of $\Xi$ given by Theorem \ref{Kuranishi family}.

Define the map
\begin{align}
\phi_{\Xi}: \mathcal{B}_{\Xi}\rightarrow \mathcal{O}\mathfrak{p}_{\Sigma}(\mathrm{G})
\end{align}
which sends $b\in \mathcal{B}_{\Xi}$ to the $\Sigma$-marked $\mathrm{G}$-oper structure on the fiber of the universal deformation over $b\in \mathcal{B}_{\Xi}.$  As the structure group of the $\Sigma$-marked deformation of $X$ is $\textnormal{Diff}_{0}(\Sigma),$ the map $\phi_{\Xi}$ is well defined.  The proof proceeds in four steps which we enumerate below.
\begin{enumerate}
\item Up to shrinking $B_{\Xi},$ the map $\phi_{\Xi}$ is injective.

          Suppose there exists $b_{1}, b_{2}\in \mathcal{B}_{\Xi}$ such that $\phi_{\Xi}(b_{1})=\phi_{\Xi}(b_{2}).$  Let $\Xi_{b_{1}}$ and $\Xi_{b_{2}}$ the $\Sigma$-marked $\mathrm{G}$-opers lying over $b_{1}, b_{2}\in \mathcal{B}_{\Xi}.$  
          
By the assumption $\phi_{\Xi}(b_{1})=\phi_{\Xi}(b_{2})$ and the definition of $\mathcal{O}\mathfrak{p}_{\Sigma}(\mathrm{G})$, there is an isomorphism $F$ of $\Sigma$-marked $\mathrm{G}$-opers between $\Xi_{b_{1}}$ and $\Xi_{b_{2}}.$  Since the family over $(\mathcal{B}_{\Xi}, 0_{\Xi})$ is universal, potentially shrinking $\mathcal{B}_{\Xi},$ the isomorphism $F$ is induced by an automorphism of $\Xi.$    

Since $\mathrm{G}$ is a complex simple Lie group of adjoint type and $X$ has no non-trivial automorphisms isotopic to the identity, the results of \cite{BD05} imply that there are no non-trivial automorphisms of $\Xi.$  Therefore, the only possibility is that $b_{1}=b_{2}$ and $F=\textnormal{id.}$  This proves that $\phi_{\Xi}$ is injective.

\item The collection $\{\phi_{\Xi}(B_{\Xi})\}_{\Xi\in \widetilde{\mathcal{O}\mathfrak{p}}_{\Sigma}(\mathrm{G})}$\footnote{Strictly speaking, we should choose a subset of objects in $\widetilde{\mathcal{O}\mathfrak{p}}_{\Sigma}(\mathrm{G})$ which is surjective onto $\mathcal{O}\mathfrak{p}_{\Sigma}(\mathrm{G}).$  We make no further mention of this set-theoretic issue, assuming the appropriate strengthening of the axiom of choice which makes this choice possible.} forms the basis of a topology on the set $\mathcal{O}\mathfrak{p}_{\Sigma}(\mathrm{G}).$  

Recall that given a set $S,$ a subset $\mathfrak{B}$ of the power set of $S$ is a basis for a topology $\mathfrak{T}$ if and only if the elements of $\mathfrak{B}$ cover $S$, and for every $U, V\in \mathfrak{B}$ and any $s\in U\cap V,$ there exists $s\in W\subset U\cap V$ such that $W\in \mathfrak{B}.$

That the collection $\{ \phi_{\Xi}(\mathcal{B}_{\Xi})\}_{\Xi\in \widetilde{\mathcal{O}\mathfrak{p}}_{\Sigma}(\mathrm{G})}$ covers $\mathcal{O}\mathfrak{p}_{\Sigma}(\mathrm{G})$ is obvious.  The second condition follows from the fact that the restriction of a universal deformation of $\Xi$ with base $\mathcal{B}$ to an open subset $\mathcal{B}^{\prime}\subset \mathcal{B}$ such that $0\in \mathcal{B}^{\prime}$ is still a universal deformation of $\Xi.$  Therefore, the collection $\{\phi_{\Xi}(B_{\Xi})\}_{\Xi\in \widetilde{\mathcal{O}\mathfrak{p}}_{\Sigma}(\mathrm{G})}$ forms the base for a topology on the set $\mathcal{O}\mathfrak{p}_{\Sigma}(\mathrm{G})$.

\item The collection $\{\phi_{\Xi}, B_{\Xi}\}_{\Xi\in \widetilde{\mathcal{O}\mathfrak{p}}_{\Sigma}(\mathrm{G})}$ forms a holomorphic atlas on the topological space $\mathcal{O}\mathfrak{p}_{\Sigma}(\mathrm{G}).$  Upon proving this, we obtain a (potentially non-Hausd\"orff) complex manifold structure on $\mathcal{O}\mathfrak{p}_{\Sigma}(\mathrm{G}).$

Suppose that $\Omega\in \textnormal{Im}(\phi_{\Xi})\cap \textnormal{Im}(\phi_{\Xi^{\prime}})\neq \emptyset.$  We need to show that
\begin{align}
\phi_{\Xi^{\prime}}^{-1}\circ \phi_{\Xi}: \phi_{\Xi}^{-1}\left(\textnormal{Im}(\phi_{\Xi})\cap \textnormal{Im}(\phi_{\Xi^{\prime}})\right) \rightarrow  \phi_{\Xi^{\prime}}^{-1}\left(\textnormal{Im}(\phi_{\Xi})\cap \textnormal{Im}(\phi_{\Xi^{\prime}})\right)
\end{align}
is holomorphic.

Let 
\begin{align}
(U,0_{\Xi}):=\phi_{\Xi}^{-1}\left(\textnormal{Im}(\phi_{\Xi})\cap \textnormal{Im}(\phi_{\Xi^{\prime}})\right)\subset (\mathcal{B}_{\Xi},0_{\Xi})
\end{align}
and
\begin{align}
(V, 0_{\Xi^{\prime}}):=\phi_{\Xi^{\prime}}^{-1}\left(\textnormal{Im}(\phi_{\Xi})\cap \textnormal{Im}(\phi_{\Xi^{\prime}})\right)\subset (\mathcal{B}_{\Xi^{\prime}},0_{\Xi^{\prime}})
\end{align}
where 
\begin{align}
\Omega=\phi_{\Xi}(0_{\Xi})=\phi_{\Xi^{\prime}}(0_{\Xi^{\prime}})
\end{align}

By construction, $(V, 0_{\Xi^{\prime}})$ is the base of a universal deformation of $\Omega.$  Since $(U,0_{\Xi})$ is the base of another deformation of $\Omega,$ there exists a unique holomorphic map
\begin{align}
f: (U,0_{\Xi})&\rightarrow (V, 0_{\Xi^{\prime}}) \\
    0_{\Xi}&\mapsto 0_{\Xi^{\prime}},
\end{align}
defined by the fact that $(V, 0_{\Xi^{\prime}})$ is the base of a universal deformation of $\Omega.$
By construction, $f=\phi_{\Xi^{\prime}}^{-1}\circ \phi_{\Xi}.$
This completes the proof.

\item The topology on $\mathcal{O}\mathfrak{p}_{\Sigma}(\mathrm{G})$ is Hausd\"orff.

By construction, the map $\mathcal{O}\mathfrak{p}_{\Sigma}(\mathrm{G})\rightarrow \mathcal{T}_{\Sigma}$ is continuous, and by Theorem \ref{param opers 2}, the fibers of this map are Hausdorff with respect to the subspace topology: they are biholomorphic to $\bigoplus_{i=1}^{\ell} \textnormal{H}^{0}(X, \mathcal{K}^{m_{i}+1})$.  Since $\mathcal{T}_{\Sigma}$ is Hausd\"orff, this implies that $\mathcal{O}\mathfrak{p}_{\Sigma}(\mathrm{G})$ is Hausd\"orff.
\end{enumerate}
This completes the proof that $\mathcal{O}\mathfrak{p}_{\Sigma}(\mathrm{G})$ is a Hausd\"orff complex manifold.  

For any $\Sigma$-marked $\mathrm{G}$-oper $\Xi,$ the base $\mathcal{B}_{\Xi}$ of any universal deformation is an open subset of the complex vector space $\mathbb{H}^{1}(X, \mathcal{A}^{\bullet})$, which by Theorem \ref{Kuranishi family} is of dimension $\textnormal{dim}_{\mathbb{C}}(G)(g-1)+(3g-3).$ 

This establishes the equality
\begin{align}
\textnormal{dim}_{\mathbb{C}}(\mathcal{O}\mathfrak{p}_{\Sigma}(\mathrm{G}))=\textnormal{dim}_{\mathbb{C}}(G)(g-1)+(3g-3).
\end{align}

Finally, our construction of the holomorphic structure on $\mathcal{O}\mathfrak{p}_{\Sigma}(\mathrm{G})$ implies that the derivative of the map $\mathcal{P}: \mathcal{O}\mathfrak{p}_{\Sigma}(\mathrm{G})\rightarrow \mathcal{T}_{\Sigma}$ identifies with the induced map in hyper-cohomology 
\begin{align}\label{derivative map}
\mathbb{H}^{1}(X, \mathcal{A}^{\bullet})\rightarrow \textnormal{H}^{1}(X, \Theta)
\end{align}
  of the map of complexes
\begin{center}
\begin{tikzcd}
\mathcal{A}^{\bullet}:=\mathcal{A}(E_{\mathrm{B}})\arrow{r} \arrow{d}{\sigma_{\mathrm{B}}} 
& \mathcal{K}\otimes \mathcal{E}_{\mathrm{B}}[\mathfrak{g}^{-1}] \arrow{d} \\ \label{induced map}
\Theta \arrow{r}
& 0.
\end{tikzcd}
\end{center}

As the kernel of the aforementioned map is given by the complex
\begin{align}
\mathcal{A}_{0}^{\bullet}:= \mathcal{E}_{\mathrm{B}}[\mathfrak{b}]\rightarrow \mathcal{K}\otimes \mathcal{E}_{\mathrm{B}}[\mathfrak{g}^{-1}], 
\end{align}
the map \eqref{derivative map} extends to an exact sequence
\begin{align}
\mathbb{H}^{1}(X, \mathcal{A}^{\bullet})\rightarrow \textnormal{H}^{1}(X, \Theta)\rightarrow
\mathbb{H}^{2}(X, \mathcal{A}_{0}^{\bullet}).
\end{align}
The hyper-cohomology group $\mathbb{H}^{2}(X, \mathcal{A}_{0}^{\bullet})$ vanishes by Theorem \ref{Kuranishi family}.  Therefore, the map $\mathcal{P}: \mathcal{O}\mathfrak{p}_{\Sigma}(\mathrm{G})\rightarrow \mathcal{T}_{\Sigma}$ has complex linear surjective derivative at every point, and thus $\mathcal{P}$ is a holomorphic submersion (see Corollary \ref{cor sur}).  This completes the proof.
\end{proof}

The proof of Theorem \ref{global opers} also implies the following result.

\begin{theorem}
The space $\mathcal{O}\mathfrak{p}_{\Sigma}(\mathrm{G})$ is the base of a universal family of $\Sigma$-marked $\mathrm{G}$-opers.  Therefore, $\mathcal{O}\mathfrak{p}_{\Sigma}(\mathrm{G})$ is a fine moduli space.
\end{theorem}

\begin{proof}
The universal families lying over the bases $\mathcal{B}_{\Xi}$ used to construct charts on $\mathcal{O}\mathfrak{p}_{\Sigma}(\mathrm{G})$ glue together to yield a universal family of $\Sigma$-marked $\mathrm{G}$-opers over $\mathcal{O}\mathfrak{p}_{\Sigma}(\mathrm{G}).$
\end{proof}

As another consequence of Theorem \ref{global opers} and the discussion of $\Sigma$-marked developed $\mathrm{G}$-opers in Section \ref{g opers} we obtain the following result.

\begin{theorem}
Let $\mathrm{G}$ be a complex simple Lie group of adjoint type.  Then the space of $\Sigma$-marked developed $\mathrm{G}$-opers has the structure of a Hausd\"{o}rff complex manifold such that the set map
\begin{align}
\mathcal{D}:  \mathcal{DO}\mathfrak{p}_{\Sigma}(\mathrm{G})\rightarrow\mathcal{O}\mathfrak{p}_{\Sigma}(\mathrm{G})
\end{align}
is a biholmorphism of complex manifolds.

Furthermore, the mapping class group $\textnormal{Mod}(\Sigma)$-action on each side is holomorphic and $\mathcal{D}$ is mapping class group equivariant.
\end{theorem}

\subsection{Marked $\mathrm{G}$-opers and the Hitchin base}

This very short section generalizes the identification of complex projective structures with the cotangent bundle of Teichmuller space from Theorem \ref {Hubbard} to the setting of $\mathrm{G}$-opers.  As always in this discussion, $\mathrm{G}$ is a complex simple Lie group of adjoint type.  

If $\mathrm{G}$ is a complex simple Lie group of adjoint type with Lie algebra $\mathfrak{g},$ let $\mathcal{B}_{\Sigma}(\mathrm{G})$ be the bundle over Teichm\"uller space whose fiber 
over a $\Sigma$-marked Riemann surface $X$ is the space $\bigoplus_{i=1}^{\ell} \textnormal{H}^{0}(X, \mathcal{K}^{m_{i}+1}),$ and where the positive integers $\{m_{i}\}_{1}^{\ell}$ are the exponents of $\mathfrak{g}.$  The set $\mathcal{B}_{\Sigma}(\mathrm{G})$ has the structure of a holomorphic vector bundle over $\mathcal{T}_{\Sigma}.$ 

\begin{theorem}\label{h base}
Let $\mathrm{G}$ be a complex simple Lie group of adjoint type.  For every $C^{\infty}$-section $s$ of the projection
\begin{align}
\pi: \mathcal{O}\mathfrak{p}_{\Sigma}(\textnormal{PSL}_{2}(\mathbb{C}))\rightarrow 
\mathcal{T}_{\Sigma},
\end{align}
there is a commutative diagram
\begin{center}
\begin{tikzcd}
\mathcal{O}\mathfrak{p}_{\Sigma}(\mathrm{G})\arrow{r}{\phi_{s}}\arrow{d}{\mathcal{P}}
& \mathcal{B}_{\Sigma}(\mathrm{G}) \arrow{d} \\
\mathcal{T}_{\Sigma} \arrow{r}{\textnormal{id.}}
& \mathcal{T}_{\Sigma},
\end{tikzcd}
\end{center}
where $\phi_{s}$ is a diffeomorphism.
If $s$ is holomorphic, then the diffeomorphism $\phi_{s}$ is holomorphic.
\end{theorem}

\textbf{Examples}:  The most obvious section $s_{\mathcal{F}}$ of $\pi: \mathcal{O}\mathfrak{p}_{\Sigma}(\textnormal{PSL}_{2}(\mathbb{C}))\rightarrow 
\mathcal{T}_{\Sigma}$ is given by selecting the Fuchsian uniformizing $\textnormal{PSL}(2,\mathbb{C})$-oper from Section \ref{models} in every fiber.  The section $s_{\mathcal{F}}$ is not holomorphic, and yields a diffeomorphism
\begin{align}
\phi_{s_{\mathcal{F}}}:\mathcal{O}\mathfrak{p}_{\Sigma}(\mathrm{G})\rightarrow \mathcal{B}_{\Sigma}(\mathrm{G})
\end{align}
which maps the sub-manifold of Fuchsian uniformizing $\mathrm{G}$-opers onto the zero section of $\mathcal{B}_{\Sigma}(\mathrm{G}).$

Holomorphic sections of $\pi: \mathcal{O}\mathfrak{p}_{\Sigma}(\textnormal{PSL}_{2}(\mathbb{C}))\rightarrow 
\mathcal{T}_{\Sigma}$ can be obtained utilizing Bers' simultaneous uniformization theorem (see the remark follow Theorem \ref{Hubbard}) which therefore yields a family of biholomorphisms onto $\mathcal{B}_{\Sigma}(\mathrm{G})$
parameterized by $\mathcal{T}_{\Sigma}.$    

\begin{proof}
Let $s$ be a section of $\pi$ and $X\in \mathcal{T}_{\Sigma}.$  Using Theorem \ref{param opers 2}, This induces a bijection (using the base-point $s(X)),$ 
\begin{align}
\mathcal{O}\mathfrak{p}_{X}(\mathrm{G})\simeq \mathcal{B}_{X}(\mathrm{G}).
\end{align}
By Theorem \ref{unique iso},  the dependence of this isomorphism on $X$ is determined by the regularity of the section $s$\footnote{Here, the real justification comes from inspecting the proof in \cite{BD05}: it is clear from the explicit construction of the map in \cite{BD05} that the regularity of $\phi_{s}$ matches the regularity of $s.$}.  This completes the proof.
\end{proof}

\section{The holonomy map and pre-symplectic geometry} \label{hol map pre}

\subsection{The holonomy map}
Finally, we come to the study of the forgetful map from $\mathcal{O}\mathfrak{p}_{\Sigma}(\mathrm{G})$ 
to the space of $C^{\infty}$-flat $\mathrm{G}$-bundles on $\Sigma,$ and prove that it is a holomorphic immersion when $\mathrm{G}$ is a complex simple Lie group of adjoint type. 

Consider the category $\widetilde{\mathcal{F}}_{\Sigma}(\mathrm{G})$ whose objects are $C^{\infty}$-flat bundles $(E_{\mathrm{G}}, \omega)$ over $\Sigma.$  Morphisms in this category are given by commutative diagrams
\[
\begin{tikzcd}
E_{\mathrm{G}} \arrow{d} \arrow{r}{\phi} 
& E_{\mathrm{G}}^{\prime} \arrow{d} \\
\Sigma \arrow{r}{h}
& \Sigma
\end{tikzcd}
\]
such that the $C^{\infty}$-map $h: \Sigma\rightarrow \Sigma$ is isotopic to the identity and $\phi$ is a smooth isomorphism of $\mathrm{G}$-bundles such that $\phi^{\star}\omega^{\prime}=\omega.$ 

It is well known that the natural topology on the set of isomorphism classes in $\widetilde{\mathcal{F}}_{\Sigma}(\mathrm{G})$ is non-Hausd\"orff, but upon restricting to a suitable sub-category we can remedy this situation.

Consider the full subcategory $\widetilde{\mathcal{F}}_{\Sigma}^{\star}(\mathrm{G})$ whose objects consist of irreducible flat $\mathrm{G}$-bundles 
whose automorphism group is equal to the center of $\mathrm{G}$. Let $\mathcal{F}_{\Sigma}^{\star}(\mathrm{G})$ denote the set of isomorphism classes. The following theorem is well known (see \cite{GOL84}), though the discussion there deals with the equivalent question for homomorphisms $\pi\rightarrow G.$

\begin{theorem}
Let $\mathrm{G}$ be a connected complex semi-simple Lie group.  Then, the set $\mathcal{F}_{\Sigma}^{\star}(\mathrm{G})$ admits the structure of a Hausd\"orff complex manifold of dimension $(2g-2)\textnormal{dim}_{\mathbb{C}}(\mathrm{G})$ where $g$ is the genus of $\Sigma.$  
\end{theorem}

The following result of Beilinson-Drinfeld \cite{BD05} allows us to only consider the complex manifold $\mathcal{F}_{\Sigma}^{\star}(\mathrm{G}).$

\begin{proposition}\label{autos}
Let $\mathrm{G}$ be a connected complex semi-simple Lie group and suppose $\Xi:=(E_{\mathrm{G}}, E_{\mathrm{B}}, \omega, X)$ is a $\Sigma$-marked $\mathrm{G}$-oper.  Then,
\begin{enumerate}
\item The automorphism group of $\Xi$ is equal to the center of $\mathrm{G}.$
\item The induced $C^{\infty}$-flat $\mathrm{G}$-bundle $(E_{\mathrm{G}}, \omega)$ on $\Sigma$ is irreducible with automorphism group equal to the center of $\mathrm{G}.$  
\end{enumerate}
\end{proposition}

Given a $\Sigma$-marked $\mathrm{G}$-oper $(E_{\mathrm{G}}, E_{\mathrm{B}}, \omega, X)$, let $(E_{\mathrm{G}}, \omega)$ denote the corresponding $C^{\infty}$-flat $\mathrm{G}$-bundle.   

By Proposition \ref{autos}, the functor
\begin{align}
\widetilde{\textnormal{H}}:\widetilde{\mathcal{O}\mathfrak{p}}_{\Sigma}(\mathrm{G})&\rightarrow \widetilde{\mathcal{F}}_{\Sigma}^{\star}(\mathrm{G}) \\
(E_{\mathrm{G}}, E_{\mathrm{B}}, \omega, X)&\mapsto (E_{\mathrm{G}}, \omega)
\end{align}
is fully faithful and
descends to a smooth map
\begin{align}
\textnormal{H}: \mathcal{O}\mathfrak{p}_{\Sigma}(\mathrm{G})\rightarrow \mathcal{F}_{\Sigma}^{\star}(\mathrm{G}).
\end{align}

We now prove the main theorem of this article.

\begin{theorem}\label{hol immersion}
Let $\mathrm{G}$ be a complex simple Lie group of adjoint type.  The map
\begin{align}
\textnormal{H}:\mathcal{O}\mathfrak{p}_{\Sigma}(\mathrm{G})\rightarrow \mathcal{F}_{\Sigma}^{\star}(\mathrm{G})
\end{align}
is a holomorphic immersion. 

Moreover, there is a (natural in $\mathrm{G}$) commutative diagram
of holomorphic maps 
\begin{center}
\begin{tikzcd}
\mathcal{O}\mathfrak{p}_{\Sigma}(\textnormal{PSL}_{2}(\mathbb{C})) \arrow{r}{\textnormal{H}} \arrow{d}{\iota_{\mathrm{G}}} 
& \mathcal{F}_{\Sigma}^{\star}(\textnormal{PSL}_{2}(\mathbb{C})) \arrow{d}{\iota_{\mathrm{G}}} \\
\mathcal{O}\mathfrak{p}_{\Sigma}(\mathrm{G}) \arrow{r}{\textnormal{H}}
&\mathcal{F}_{\Sigma}^{\star}(\mathrm{G}).
\end{tikzcd}
\end{center}
\end{theorem}

\textbf{Remark:} For $G=\textnormal{PSL}_{2}(\mathbb{C})$, this was proved independently (and with varying methods) by Earle \cite{EAR81}, Hejhal \cite{HEJ78} and Hubbard \cite{HUB81}.  Our proof is in the spirit of the proof of Hubbard \cite{HUB81}: in particular we will identify the differential of $\textnormal{H}$ with a certain induced map in hyper-cohomology, and use a differential-geometric argument to prove the injectivity of this map.

\begin{proof}
Let $\Xi:=(E_{\mathrm{G}}, E_{\mathrm{B}}, \omega, X)$ be a $\Sigma$-marked $\mathrm{G}$-oper and recall the two term complex
\begin{align}
\mathcal{A}^{\bullet}:=\mathcal{A}(E_{\mathrm{B}})\xrightarrow{[\hat{\omega}, \ ]} \mathcal{K}\otimes \mathcal{E}_{\mathrm{B}}[\mathfrak{g}^{-1}],
\end{align}
which controls the deformation theory of $\Xi.$  

First, we show that there is an injective map

\begin{align}\label{oper to flat}
\mathcal{A}^{\bullet}\rightarrow \mathcal{B}^{\bullet}
\end{align}
where 
\begin{align}
\mathcal{B}^{\bullet}:= \mathcal{E}_{\mathrm{G}}[\mathfrak{g}]\xrightarrow{[\hat{\omega}, \ ]} \mathcal{K}\otimes \mathcal{E}_{\mathrm{G}}[\mathfrak{g}]
\end{align}
is the holomorphic de-Rham complex of the holomorphic flat bundle $(E_{\mathrm{G}}, \omega, X).$

To define \eqref{oper to flat}, consider the commutative diagram
\begin{equation}
\begin{tikzcd}
\mathcal{A}(E_{\mathrm{B}})\arrow{d}{\iota} \arrow{dr}{\sigma_{\mathrm{B}}} \\
\mathcal{A}(E_{\mathrm{G}}) \arrow{r}{\sigma_{\mathrm{G}}}
& \Theta.
\end{tikzcd}
\end{equation}
where $\iota$ is the inclusion.

If $\hat{\omega}: \Theta\rightarrow \mathcal{A}(E_{\mathrm{G}})$ is the holomorphic flat connection appearing in $\Xi$, define the injective map
\begin{align}\label{to flat bundle}
\iota-\hat{\omega}\circ \sigma_{\mathrm{B}}: \mathcal{A}(E_{\mathrm{B}})\rightarrow \mathcal{A}(E_{\mathrm{G}}).
\end{align}
Since $\hat{\omega}$ is a splitting, $\sigma_{\mathrm{G}}\circ(\iota-\hat{\omega}\circ \sigma_{\mathrm{B}})=0$.

Since $\mathcal{E}_{\mathrm{G}}[\mathfrak{g}]=\textnormal{ker}(\sigma_{\mathrm{G}}),$ the map \eqref{to flat bundle} lifts to a map
\begin{align}
\Phi: \mathcal{A}(E_{\mathrm{B}})\rightarrow \mathcal{E}_{\mathrm{G}}[\mathfrak{g}].
\end{align}

Define the injective map \eqref{oper to flat} by 
\begin{equation}
\begin{tikzcd}
 \mathcal{A}(E_{\mathrm{B}}) \arrow{d}{\Phi} \arrow{r}{[\hat{\omega}, - ]}
& \mathcal{K}\otimes \mathcal{E}_{\mathrm{B}}[\mathfrak{g}^{-1}] \arrow{d} \\
\mathcal{E}_{\mathrm{G}}[\mathfrak{g}]  \arrow{r}{[\hat{\omega}, - ]}
& \mathcal{K} \otimes \mathcal{E}_{\mathrm{G}}[\mathfrak{g}],
\end{tikzcd}
\end{equation}
where the right vertical arrow is the inclusion of $\mathcal{K}\otimes \mathcal{E}_{\mathrm{B}}[\mathfrak{g}^{-1}]$ as a locally-free sub-sheaf.

Viewing the complexes $\mathcal{A}^{\bullet}$ and $\mathcal{B}^{\bullet}$ as objects in the abelian category of bounded complexes of coherent analytic sheaves over $X,$  the map \eqref{oper to flat} has a co-kernel $\mathcal{N}^{\bullet},$ and hence there is an exact sequence of complexes
\begin{align}
0\rightarrow \mathcal{A}^{\bullet}\rightarrow \mathcal{B}^{\bullet}\rightarrow \mathcal{N}^{\bullet}\rightarrow 0.
\end{align}

Fortunately, we can identify an explicit model of the complex $\mathcal{N}^{\bullet}$ consisting of locally-free sheaves.  Let $\Psi: \Theta\rightarrow \mathcal{E}_{\mathrm{B}}[\mathfrak{g}/\mathfrak{b}]$ be the second fundamental form of $\omega$ relative to $E_{\mathrm{B}}$ and $\mathcal{N}^{0}:=\textnormal{coker}(\Psi).$  Since $\Psi$ is an injective morphism of the corresponding bundles, $\mathcal{N}^{0}$ is a holomorphic vector bundle on $X.$ 

We now show that there is an exact sequence
\begin{align}\label{seq nat}
0\rightarrow \mathcal{A}(E_{\mathrm{B}})\xrightarrow{\Phi}  \mathcal{E}_{\mathrm{G}}[\mathfrak{g}]\xrightarrow{p} \mathcal{N}^{0}\rightarrow 0.
\end{align}

Consider the commutative diagram with exact top row
\begin{center}
\begin{tikzcd}
&
0
& 
&
& \\
0 
& \mathcal{N}^{0} \arrow{l} \arrow{u}
& \mathcal{E}_{\mathrm{B}}[\mathfrak{g}/\mathfrak{b}] \arrow{l} 
& \Theta \arrow[swap]{l}{\Psi}
& 0 \arrow{l} \\

&\mathcal{E}_{\mathrm{G}}[\mathfrak{g}] \arrow{u}{p} 
& \mathcal{A}(E_{\mathrm{G}}) \arrow{ur}{\sigma_{\mathrm{G}}} \arrow[dashrightarrow]{u}
&
& \\
&
& \mathcal{A}(E_{\mathrm{B}}) \arrow{ul}{\Phi} \arrow{u}{\iota} \arrow[swap]{uur}{\sigma_{\mathrm{B}}}
&
& \\
&
& 0 \arrow{u}.
&
&
\end{tikzcd}
\end{center}
The map $p:\mathcal{E}_{\mathrm{G}}[\mathfrak{g}]\rightarrow \mathcal{N}^{0}$ is defined as the composition of the surjective projections
\begin{align}
\mathcal{E}_{\mathrm{G}}[\mathfrak{g}]\rightarrow \mathcal{E}_{\mathrm{B}}[\mathfrak{g}/\mathfrak{b}]\rightarrow \mathcal{N}^{0},
\end{align}
and therefore $p$ is surjective.

The vertical dashed arrow 
\begin{align}
\mathcal{A}(E_{\mathrm{G}})\dashrightarrow \mathcal{E}_{\mathrm{B}}[\mathfrak{g}/\mathfrak{b}]
\end{align}
is defined as the composition $\Psi\circ \sigma_{\mathrm{G}}$ so that the diagram commutes.

Note that
\begin{align}
0\rightarrow \mathcal{A}(E_{\mathrm{B}})\xrightarrow{\iota} \mathcal{A}(E_{\mathrm{G}})\dashrightarrow \mathcal{E}_{\mathrm{B}}[\mathfrak{g}/\mathfrak{b}]
\end{align}
is \emph{not} exact in the middle, though the initial map $\iota$ is injective.  Since $\Psi$ is injective, there \emph{is} an exact sequence
\begin{align}
0\rightarrow \mathcal{E}_{\mathrm{G}}[\mathfrak{g}]\rightarrow \mathcal{A}(E_{\mathrm{G}})\dashrightarrow \mathcal{E}_{\mathrm{B}}[\mathfrak{g}/\mathfrak{b}].
\end{align}

By commutativity and the exactness of the top row,
$\mathcal{A}(E_{\mathrm{B}})\subset \textnormal{ker}(p).$  
Since $p$ is a surjective map of holomorphic
vector bundles, $\textnormal{ker}(p)$ is a holomorphic vector bundle such that
\begin{align}
 \textnormal{rk}(\textnormal{ker}(p))=\textnormal{rk}(E_{\mathrm{G}}[\mathfrak{g}])-\textnormal{rk}(\mathcal{N}^{0}).
 \end{align}
Now we compute,
\begin{align}
\textnormal{rk}(\mathcal{A}(E_{\mathrm{B}}))&=\textnormal{rk}(E_{\mathrm{B}}[\mathfrak{b}])+\textnormal{rk}(\Theta) \\
&= \textnormal{rk}(E_{\mathrm{G}}[\mathfrak{g}])-\textnormal{rk}(E_{\mathrm{B}}[\mathfrak{g}/\mathfrak{b}])+\textnormal{rk}(\Theta) \\
&=\textnormal{rk}(E_{\mathrm{G}}[\mathfrak{g}])-\textnormal{rk}(\mathcal{N}^{0}) \\
&= \textnormal{rk}(\textnormal{ker}(p)).
\end{align}
Therefore, since $\mathcal{A}(E_{\mathrm{B}})\subset \textnormal{ker}(p),$ this implies that $\mathcal{A}(E_{\mathrm{B}})= \textnormal{ker}(p).$  This defines the promised exact sequence \eqref{seq nat}
\begin{align}
0\rightarrow \mathcal{A}(E_{\mathrm{B}})\xrightarrow{\Phi}  \mathcal{E}_{\mathrm{G}}[\mathfrak{g}]\xrightarrow{p} \mathcal{N}^{0}\rightarrow 0.
\end{align}

Defining $\mathcal{N}^{1}:= \mathcal{K}\otimes \mathcal{E}_{B}[\mathfrak{g}/\mathfrak{g}^{-1}],$ we obtain a short exact sequence of complexes
\[
\begin{tikzcd}
0\arrow{d} 
& 0\arrow{d} \\
\mathcal{A}(E_{\mathrm{B}}) \arrow{d}{\Phi} \arrow{r}{[\hat{\omega}, -]}
& \mathcal{K}\otimes \mathcal{E}_{\mathrm{B}}[\mathfrak{g}^{-1}] \arrow{d} \\
\mathcal{E}_{\mathrm{G}}[\mathfrak{g}] \arrow{d}{p} \arrow{r}{[\hat{\omega}, -]}
& \mathcal{K} \otimes \mathcal{E}_{\mathrm{G}}[\mathfrak{g}] \arrow{d}\\
\mathcal{N}^{0} \arrow{r} \arrow{d}
& \mathcal{K}\otimes \mathcal{E}_{B}[\mathfrak{g}/\mathfrak{g}^{-1}] \arrow{d} \\
0
&0,
\end{tikzcd}
\]
where the bottom horizontal arrow is uniquely defined by exactness and commutativity.

Define the complex $\mathcal{N}^{\bullet}$ via
\begin{align}
\mathcal{N}^{\bullet}:= \mathcal{N}^{0}\rightarrow \mathcal{K}\otimes \mathcal{E}_{B}[\mathfrak{g}/\mathfrak{g}^{-1}].
\end{align}

The holomorphic De-Rham complex $\mathcal{B}^{\bullet}$ is a resolution of the local system $\mathbb{E}_{G}[\mathfrak{g}]_{\omega}$ defined by the holomorphic flat connection $\omega.$  Therefore, there is a canonical isomorphism $\textnormal{H}^i(X, \mathbb{E}_{G}[\mathfrak{g}]_{\omega})\simeq \mathbb{H}^{i}(X, \mathcal{B}^{\bullet}).$  

Taking the relevant chunk of the long exact sequence in hyper-cohomology yields,
\begin{align}
...\rightarrow \mathbb{H}^{0}(X, \mathcal{N}^{\bullet}) \rightarrow \mathbb{H}^{1}(X, \mathcal{A}^{\bullet})\rightarrow \textnormal{H}^{1}(X, \mathbb{E}_{G}[\mathfrak{g}]_{\omega})\rightarrow...
\end{align}
If $\mathbb{H}^{0}(X, \mathcal{N}^{\bullet})=\{0\}$, then the $\mathbb{C}$-linear map 
\begin{align}
\mathbb{H}^{1}(X, \mathcal{A}^{\bullet})\rightarrow \textnormal{H}^{1}(X, \mathbb{E}_{G}[\mathfrak{g}]_{\omega})
\end{align}
is injective.  But, there is a canonical commutative diagram of $\mathbb{C}$-linear maps
\[
\begin{tikzcd}
T_{[(E_{\mathrm{G}}, E_{\mathrm{B}}, \omega, X)]}\mathcal{O}\mathfrak{p}_{\Sigma}(\mathrm{G}) \arrow{r}{d\textnormal{H}} \arrow{d}
& T_{H([(E_{\mathrm{G}}, E_{\mathrm{B}}, \omega, X)])} \mathcal{F}_{\Sigma}^{\star}(\mathrm{G}) \arrow{d} \\
\mathbb{H}^{1}(X, \mathcal{A}^{\bullet}) \arrow{r} 
& \textnormal{H}^{1}(X, \mathbb{E}_{G}[\mathfrak{g}]_{\omega}),
\end{tikzcd}
\]
where the vertical arrows are $\mathbb{C}$-linear isomorphisms.  This proves that $d\textnormal{H}$ is $\mathbb{C}$-linear.  Therefore, $\textnormal{H}$ is holomorphic.  Hence, if $\mathbb{H}^{0}(X, \mathcal{N}^{\bullet})=\{0\},$ it follows that $\textnormal{H}$ is a holomorphic immersion.
   
To show that $\mathbb{H}^{0}(X, \mathcal{N}^{\bullet})=\{0\},$ we will verify the apriori stronger vanishing $\textnormal{H}^{0}(X, \mathcal{N}^{0})=\{0\}.$  Given the short exact sequence
\begin{align}\label{normal LES}
0\rightarrow \Theta\xrightarrow{\Psi} \mathcal{E}_{\mathrm{B}}[\mathfrak{g}/\mathfrak{b}]\rightarrow \mathcal{N}^{0}\rightarrow 0
\end{align}
and using the fact that $\textnormal{H}^{0}(X, \mathcal{E}_{\mathrm{B}}[\mathfrak{g}/\mathfrak{b}])=0$ \cite{AB83}[pg. 592], the vanishing of $\textnormal{H}^{0}(X, \mathcal{N}^{0})$ is equivalent to the injectivity of the map
\begin{align}\label{injective map}
\textnormal{H}^{1}(X, \Theta)\rightarrow \textnormal{H}^{1}(X, \mathcal{E}_{\mathrm{B}}[\mathfrak{g}/\mathfrak{b}])
\end{align}
induced by $\Psi$ arising in the long exact sequence of \eqref{normal LES}.  

To achieve this, recall from section \ref{models} that there is a $C^{\infty}$-bundle isomorphism
\begin{align}
E_{\mathrm{G}}[\mathfrak{g}]\simeq\bigoplus_{i=-m^{\ell}}^{m_{\ell}} K^{i}\otimes \mathfrak{g}_{i}.
\end{align}
Furthermore, in these coordinates the holomorphic structure on $E_{\mathrm{G}}[\mathfrak{g}]$ is defined by the following $\overline{\partial}$-operator (see Section \ref{models}):
\begin{align}\label{partial op}
\overline{\partial}\left(\sum_{i=-m_{\ell}}^{m_{\ell}} \beta_{i}\otimes V_{i}\right)=
\sum_{i=-m}^{m} \overline{\partial}_{i} \beta_{i}\otimes V_{i}+h\cdot \beta_{i}\otimes\textnormal{ad}(e_{1})(V_{i}).
\end{align}
As in Section \ref{models}, $h$ is the hermitian metric on $\Theta$ arising from the uniformizing hyperbolic metric on $X,$ 
and $\overline{\partial}_{i}$ is the $\overline{\partial}$-operator defining the holomorphic structure on $i$-th pluri-canonical bundle $K^{i}.$  Finally $e_{1}\in \mathfrak{g}_{1}$ is a principal nilpotent element.

In these coordinates, there is a $C^{\infty}$-isomorphism
\begin{align}
E_{\mathrm{B}}[\mathfrak{g}/\mathfrak{b}]\simeq \bigoplus_{i=-m_{\ell}}^{-1} K^{i}\otimes \mathfrak{g}_{i},
\end{align}
and the previously defined $\overline{\partial}$-operator \eqref{partial op} defines a holomorphic structure on $E_{\mathrm{B}}[\mathfrak{g}/\mathfrak{b}].$

Moreover, by Proposition \ref{unif oper}, the second fundamental form $\Psi: \Theta\rightarrow \mathcal{E}_{\mathrm{B}}[\mathfrak{g}/\mathfrak{b}]$ is given by
\begin{align}
\Theta &\rightarrow \bigoplus_{i=-m_{\ell}}^{-1} K^{i}\otimes \mathfrak{g}_{i} \\
   \xi &\mapsto \xi\otimes f_{1}.
\end{align}
Here, recall that $f_{1}\in \mathfrak{g}_{-1}$ is a principal nilpotent element such that $\{f_{1},x, e_{1}\}$ are an $\mathfrak{sl}_{2}$-triple in $\mathfrak{g}$ (see Section \ref{models}).  

We need some more Lie-theoretic preliminaries before continuing: in particular we need a basis of $\mathfrak{g}$ which is well adapted to the $\mathfrak{sl}_{2}$-triple $\{f_{1}, x, e_{1}\}$.  

The above $\mathfrak{sl}_{2}$-triple induces an injective homomorphism of Lie algebras
\begin{align}
\iota_{\mathfrak{g}}: \mathfrak{sl}_{2}(\mathbb{C})\rightarrow \mathfrak{g}.
\end{align}
With respect to the induced adjoint action of $\mathfrak{sl}_{2}(\mathbb{C})$ on $\mathfrak{g}$, the Lie algebra $\mathfrak{g}$ decomposes as a sum of simple $\mathfrak{sl}_{2}(\mathbb{C})$-modules
\begin{align}
\mathfrak{g}=\bigoplus_{i=1}^{\ell} W_{i},
\end{align}
where the dimension of $W_{i}$ is $2m_{i}+1;$ this is one way to define the exponents $\{m_{i}\}_{i=1}^{\ell}$ of $\mathfrak{g}.$ 

It is a standard fact in Lie theory that there exists $H_{i}\in W_{i}$ regular semi-simple elements, with $H_{1}=x,$ such that
\begin{align}
\{\textnormal{ad}^{j}(f_{1})(H_{i})\}_{j=1}^{m_{i}}
\end{align} 
is a basis of 
\begin{align}
W_{i}\cap \bigoplus_{k=-m_{\ell}}^{-1} \mathfrak{g}_{k}.
\end{align}
This yields a basis
\begin{align}\label{BASIS}
\{\{\textnormal{ad}^{j}(f_{1})(H_{i})\}_{i=1}^{\ell}\}_{j=1}^{m_{i}}
\end{align}
of 
\begin{align}
\mathfrak{g}/\mathfrak{b}\simeq \bigoplus_{k=-m_{\ell}}^{-1} \mathfrak{g}_{k}.
\end{align}

With these Lie theoretic preliminaries out of the way, we utilize the Dolbeault resolution and assume there exists $\mu\in \mathcal{A}^{(0,1)}(X, \Theta_{X})$ such that 
\begin{align}
\Psi(\mu)=\mu\otimes f_{1}=0\in \textnormal{H}^{(0,1)}(X, E_{\mathrm{B}}[\mathfrak{g}/\mathfrak{b}]).
\end{align}  
If we show that this implies that $\mu=0\in \textnormal{H}^{(0,1)}(X, \Theta),$ this will prove the injectivity of the map \ref{injective map}, thereby proving that the holonomy map $\textnormal{H}$ is an immersion.

With respect to the basis \eqref{BASIS}, 
\begin{align}
\mu\otimes f_{1}=0\in \textnormal{H}^{(0,1)}(X, E_{\mathrm{B}}[\mathfrak{g}/\mathfrak{b}])
\end{align}
if and only if there exists smooth sections 
\begin{align}
\{\beta_{i}^{j}\}_{i=1}^{\ell} \subset \mathcal{A}^{0}(X, \mathcal{K}^{-j})
\end{align}
for $1\leq j \leq m_{\ell},$
 and a smooth section
\begin{align}\label{explicit}
s=\sum_{j=1}^{m_{\ell}}\left(\sum_{i=1}^{\ell} \beta_{i}^{j}\otimes \textnormal{ad}(f_{1})^{j}(H_{i})\right)
\end{align}
of $E_{\mathrm{B}}[\mathfrak{g}/\mathfrak{b}]$
which satisfies
\begin{align}\label{zero eq}
\overline{\partial}s=\mu \otimes f_{1}.
\end{align}

Expanding \eqref{zero eq} using the explicit form \eqref{explicit} and using the induced decomposition $\mathfrak{g}/\mathfrak{b}=\bigoplus_{i=-m_{\ell}}^{-1}\mathfrak{g}_{i}$ leads to the explicit system of equations:
\begin{align}
\sum_{i=1}^{\ell} \overline{\partial}_{-1}\beta_{i}^{1}\otimes \textnormal{ad}(f_{1})(H_{i}) + h\cdot \beta_{i}^{2}\otimes [e_{1}, \textnormal{ad}(f_{1})^{2}(H_{i})]=\mu\otimes f_{1},
\end{align}
and
\begin{align}
\sum_{i=1}^{\ell} \overline{\partial}_{-j}\beta_{i}^{j} \otimes \textnormal{ad}(f_{1})^{j}(H_{i}) + h\cdot \beta_{i}^{j+1}\otimes [e_{1}, \textnormal{ad}(f_{1})^{j+1}(H_{i})]=0,
\end{align}
for all $2\leq j\leq m_{\ell}.$

We proceed by induction starting at $j=m_{\ell}.$  Since $\textnormal{ad}(f_{1})^{m_{\ell}+1}=0$, we arrive at the equation
\begin{align}
\sum_{i=1}^{\ell}  \overline{\partial}_{-m_{\ell}}\beta_{i}^{m_{\ell}} \otimes \textnormal{ad}(f_{1})^{m_{\ell}}(H_{i})=0.
\end{align}
Since $\textnormal{H}^{0}(X, \mathcal{K}^{-m_{\ell}})=\{0\},$ this implies that $\beta_{i}^{m_{\ell}}=0$ for all $1\leq i\leq \ell.$  

Continuing by induction, the fact that $\textnormal{H}^{0}(X, \mathcal{K}^{-j})=\{0\}$ for all $1\leq j\leq m_{\ell}$ implies
\begin{align}
\beta_{i}^{j}=0
\end{align}
for all $1\leq i \leq \ell$ and for all $2\leq j \leq m_{\ell}.$

Hence, we arrive at the final equation
\begin{align}
\sum_{i=1}^{\ell} \overline{\partial}_{-1}\beta_{i}^{1}\otimes \textnormal{ad}(f_{1})(H_{i})=\mu\otimes f_{1}.
\end{align}

But, recalling that $\textnormal{ad}(f_{1})(H_{1})=[f_{1}, x]=2f_{1},$ we obtain
\begin{align}
\label{eq 1}2\overline{\partial}_{-1} \beta_{1}^{1}&=\mu, \\
\sum_{2}^{\ell} \overline{\partial}_{-1} \beta_{i}^{1}\otimes \textnormal{ad}(f_{1})(H_{i})&=0.
\end{align}
Therefore, $\overline{\partial}_{-1} \beta_{i}^{1}=0$ for all $2\leq i\leq \ell$ which implies that
$\beta_{i}^{1}=0$ for all $2\leq i \leq \ell.$  
Finally, remembering that $K^{-1}\simeq \Theta,$ \eqref{eq 1} implies $\mu=0\in \textnormal{H}^{(0,1)}(X, \Theta)\simeq \textnormal{H}^{1}(X, \Theta).$   

This proves that the map \eqref{injective map} is injective and subsequently $\textnormal{H}^{0}(X, \mathcal{N}^{0})=0.$  This completes the proof that the map $\textnormal{H}$ is an immersion.

\end{proof}

Recall that the holonomy map identifies the complex manifold $\mathcal{F}_{\Sigma}^{\star}(\mathrm{G})$ with the space of conjugacy classes $\textnormal{Hom}^{\star}(\pi, G)/\mathrm{G}$ of irreducible homomorphisms with centralizer equal to the center of $G.$  Theorem \ref{hol immersion} translates into the following statement in terms of developed $\mathrm{G}$-opers.

\begin{corollary}
Let $\mathrm{G}$ be a complex simple Lie group of adjoint type.  Then the holonomy map
\begin{align}
\textnormal{H}: \mathcal{DO}\mathfrak{p}_{\Sigma}(\mathrm{G})\rightarrow \textnormal{Hom}^{\star}(\pi, G)/G
\end{align}
is a holomorphic immersion and there is a commutative diagram
\begin{center}
\begin{tikzcd}
\mathcal{DO}\mathfrak{p}_{\Sigma}(\textnormal{PSL}_{2}(\mathbb{C})) \arrow{r}{\textnormal{H}} \arrow{d}{\iota_{\mathrm{G}}}
& \textnormal{Hom}^{\star}(\pi, \textnormal{PSL}_{2}(\mathbb{C}))/\textnormal{PSL}_{2}(\mathbb{C}) \arrow{d}{\iota_{\mathrm{G}}} \\
\mathcal{DO}\mathfrak{p}_{\Sigma}(\mathrm{G}) \arrow{r}{\textnormal{H}}
&\textnormal{Hom}^{\star}(\pi, \mathrm{G})/\mathrm{G}.
\end{tikzcd}
\end{center}
\end{corollary}

\subsection{The pre-symplectic geometry of opers}

In this final section, we show that the space $\mathcal{O}\mathfrak{p}_{\Sigma}(\mathrm{G})$ admits a natural holomorphic pre-symplectic form where $\mathrm{G}$ is a complex simple Lie group of adjoint type.  Furthermore, we show that the holomorphic vector bundle $\mathcal{B}_{\Sigma}(\mathrm{G})$ over $\mathcal{T}_{\Sigma}$ has a natural holomorphic pre-symplectic form for which the isomorphism of Theorem \ref{h base} is a holomorphic pre-symplectomorphism when the defining section is holomorphic and Lagrangian.

These results are an extension of the theorem of Kawai \cite{KAW96} which states
that the bi-holomorphism
\begin{align}
\mathcal{CP}_{\Sigma} \simeq T^{\star}\mathcal{T}_{\Sigma}
\end{align}
provided by a Bers' section is a complex symplectic map (up to a constant factor).  See Loustau \cite{LOU15} for a
nice clarification/discussion of this result.

The tangent space to $\mathcal{F}_{\Sigma}^{\star}(\mathrm{G})$ at a $C^{\infty}$-flat bundle $(E_{\mathrm{G}}, \omega)$ is isomorphic to the first hyper-cohomology $\mathbb{H}^{1}(\Sigma, \mathcal{B}_{\textnormal{top}}^{\bullet})$ of the smooth De-Rham complex
\begin{align}
\mathcal{B}_{\textnormal{top}}^{\bullet}:= \mathcal{E}_{\mathrm{G}}[\mathfrak{g}]_{\textnormal{top}}\xrightarrow{[\hat{\omega}, -]} \mathcal{A}^{1}(\Sigma) \otimes \mathcal{E}_{\mathrm{G}}[\mathfrak{g}]_{\textnormal{top}}\xrightarrow{[\hat{\omega}, -]} \mathcal{A}^{2}(\Sigma)\otimes \mathcal{E}_{\mathrm{G}}[\mathfrak{g}]_{\textnormal{top}}.
\end{align}
Here $\mathcal{A}^{i}(\Sigma)$ is the sheaf of germs of smooth complex-valued differential $i$-forms on $\Sigma$ and the subscript top. indicates the relevant sheaves of germs of smooth sections.  This map is the $C^{\infty}$-analogue of the bracket defined in Theorem \ref{hol immersion}.

The smooth De-Rham complex $\mathcal{B}_{\textnormal{top}}^{\bullet}$ is a resolution of the local system $\mathbb{E}_{G}[\mathfrak{g}]_{\omega}.$ Hence, there are isomorphisms
\begin{align}
\mathbb{H}^{i}(\Sigma,\mathcal{B}_{\textnormal{top}}^{\bullet})\simeq  \textnormal{H}^{i}(X, \mathbb{E}_{G}[\mathfrak{g}]_{\omega})\simeq\mathbb{H}^{i}(X, \mathcal{B}^{\bullet}),
\end{align}
where $\mathcal{B}^{\bullet}$ is the corresponding holomorphic De-Rham complex of $(E_{\mathrm{G}}, \omega, X).$ 

Next we introduce the Atiyah-Bott-Goldman symplectic form on $\mathcal{F}_{\Sigma}^{\star}(\mathrm{G}).$ Note that there is an isomorphism
\begin{align}
T_{[(E_{\mathrm{G}}, \omega)]}\mathcal{F}_{\Sigma}^{\star}(\mathrm{G})\simeq
\textnormal{H}^{1}(\Sigma, \mathbb{E}_{G}[\mathfrak{g}]_{\omega}).
\end{align}

The cup product induces a non-degenerate, skew-symmetric $\mathbb{C}$-linear map
\begin{align}
\textnormal{H}^{1}(\Sigma, \mathbb{E}_{G}[\mathfrak{g}]_{\omega})\otimes \textnormal{H}^{1}(\Sigma, \mathbb{E}_{G}[\mathfrak{g}]_{\omega})\xrightarrow{\cup} H^{2}\left(\Sigma, \mathbb{E}_{G}[\mathfrak{g}]_{\omega}\otimes \mathbb{E}_{G}[\mathfrak{g}]_{\omega}\right).
\end{align}
Using the Killing form $\mathrm{B}$ on $\mathfrak{g}$ as a coefficient pairing defines a symmetric $\mathbb{C}$-linear map
\begin{align}
\mathrm{B}: H^{2}\left(\Sigma, \mathbb{E}_{G}[\mathfrak{g}]_{\omega}\otimes \mathbb{E}_{G}[\mathfrak{g}]_{\omega}\right)\rightarrow H^{2}(\Sigma, \mathbb{C}).
\end{align}
Finally, taking the cap product with the fundamental class of $\Sigma,$  remembering that $\Sigma$ is \emph{oriented}, defines a non-degenerate skew-symmetric $\mathbb{C}$-linear map
\begin{align}
\eta_{\mathrm{G}}: \textnormal{H}^{1}(\Sigma, \mathbb{E}_{G}[\mathfrak{g}]_{\omega})\otimes \textnormal{H}^{1}(\Sigma, \mathbb{E}_{G}[\mathfrak{g}]_{\omega})\rightarrow \mathbb{C}.
\end{align}

Following Atiyah-Bott \cite{AB83}, Goldman proved \cite{GOL84} that $\eta_{\mathrm{G}}$ defines a non-degenerate, closed holomorphic differential $2$-form on the complex manifold $\mathcal{F}_{\Sigma}^{\star}(\mathrm{G}).$  The complex symplectic form $\eta_{\mathrm{G}}$ is called the \emph{Atiyah-Bott-Goldman} symplectic form.

We now prove that for a complex simple Lie group of adjoint type, the complex manifold $\mathcal{O}\mathfrak{p}_{\Sigma}(\mathrm{G})$ admits a closed holomorphic $2$-form of constant rank.  Such a $2$-form is called a complex \emph{pre-symplectic} form.

\begin{theorem}\label{sym form}
Let $\mathrm{G}$ be a complex simple Lie group of adjoint type and equip $\mathfrak{g}$ with the Killing form.
Then, the complex manifold $\mathcal{O}\mathfrak{p}_{\Sigma}(\mathrm{G})$ admits a complex pre-symplectic form $\tau_{\mathrm{G}}$ of constant (complex) rank $6g-6$, for which the fibers of the map 
\begin{align}
\mathcal{O}\mathfrak{p}_{\Sigma}(\mathrm{G})\rightarrow \mathcal{T}_{\Sigma}
\end{align}
are maximal isotropic sub-manifolds.  

Furthermore, the natural holomorphic embedding
\begin{align}
\iota_{\mathrm{G}}:\mathcal{O}\mathfrak{p}_{\Sigma}(\textnormal{PSL}_{2}(\mathbb{C}))\rightarrow \mathcal{O}\mathfrak{p}_{\Sigma}(\mathrm{G})
\end{align} 
is a symplectic embedding which satisfies
\begin{align}
\iota_{\mathrm{G}}^{\star}\tau_{\mathrm{G}}=\tau_{\textnormal{PSL}_{2}(\mathbb{C})}.
\end{align}

Finally, the form $\tau_{\mathrm{G}}$ is non-degenerate if and only if $\mathrm{G}$ is
isomorphic to $\textnormal{PSL}_{2}(\mathbb{C}).$
\end{theorem}
\textbf{Remark:}  The induced complex symplectic structure on $\mathcal{O}\mathfrak{p}_{\Sigma}(\textnormal{PSL}(2, \mathbb{C}))$ is the usual complex symplectic structure on the moduli space of $\Sigma$-marked complex projective structures (See \cite{LOU15}). 

\begin{proof}
By Theorem \ref{hol immersion}, the map
\begin{align}
\textnormal{H}: \mathcal{O}\mathfrak{p}_{\Sigma}(\mathrm{G})\rightarrow \mathcal{F}_{\Sigma}^{\star}(\mathrm{G})
\end{align}
is a holomorphic immersion.  Therefore, $\tau_{\mathrm{G}}:=\textnormal{H}^{\star}\eta_{\mathrm{G}}$ yields a closed holomorphic $2$-form on $\mathcal{O}\mathfrak{p}_{\Sigma}(\mathrm{G}).$  By \cite{BD05}, the restriction of $\textnormal{H}$ to the fibers of 
\begin{align}\label{projection}
\mathcal{P}: \mathcal{O}\mathfrak{p}_{\Sigma}(\mathrm{G})\rightarrow \mathcal{T}_{\Sigma}
\end{align}
is a proper Lagrangian embedding.  Since Lagrangian sub-manifolds are maximal isotropic sub-manifolds and $\textnormal{H}$ is an immersion, this implies that $\tau_{\mathrm{G}}$ has constant rank.

This immediately implies that the fibers of \eqref{projection} are maximal isotropic submanifolds for $\tau_{\mathrm{G}}.$

By Theorem \ref{hol immersion}, the map
\begin{align}
\textnormal{H}: \mathcal{O}\mathfrak{p}_{\Sigma}(\textnormal{PSL}_{2}(\mathbb{C}))\rightarrow \mathcal{F}_{\Sigma}^{\star}(\textnormal{PSL}_{2}(\mathbb{C}))
\end{align}
is a local bi-holomorphism, and therefore $\tau_{\textnormal{PSL}(2,\mathbb{C})}$ is a holomorphic symplectic form.

The commutativity of the diagram
\begin{center}
\begin{tikzcd}
 \mathcal{O}\mathfrak{p}_{\Sigma}(\textnormal{PSL}_{2}(\mathbb{C})) \arrow{r}{\iota_{\mathrm{G}}} \arrow{d}{\textnormal{H}}
 & \mathcal{O}\mathfrak{p}_{\Sigma}(\mathrm{G}) \arrow{d}{\textnormal{H}} \\
 \mathcal{F}_{\Sigma}^{\star}(\textnormal{PSL}_{2}(\mathbb{C})) \arrow{r}{\iota_{\mathrm{G}}}
 & \mathcal{F}_{\Sigma}^{\star}(\mathrm{G})
 \end{tikzcd}
 \end{center}
 implies that 
 \begin{align}
 \iota_{\mathrm{G}}:\mathcal{O}\mathfrak{p}_{\Sigma}(\textnormal{PSL}_{2}(\mathbb{C}))\rightarrow \mathcal{O}\mathfrak{p}_{\Sigma}(\mathrm{G})
\end{align} 
is a symplectic embedding which satisfies
\begin{align}
\iota_{\mathrm{G}}^{\star}\tau_{\mathrm{G}}=\tau_{\textnormal{PSL}_{2}(\mathbb{C})}.
\end{align}
This proves that the rank of $\tau_{\mathrm{G}}$ is $6g-6$ and completes the proof.
\end{proof}

Given $(M, \tau)$ a pre-symplectic manifold, $M$ admits a foliation given by the (integrable) distribution $\textnormal{ker}(\tau).$  The leaf space of this foliation, if it is a manifold, is called the reduced phase space of $(M, \tau),$ and admits a canonical symplectic structure $(M^{\textnormal{red}}, \tau^{\textnormal{red}})$ such that the projection 
\begin{align}
R: M\rightarrow M^{\textnormal{red}}
\end{align}
satisfies $R^{\star}\tau^{\textnormal{red}}=\tau$

Theorem \ref{sym form} yields the following result.
\begin{corollary}\label{reduced}
The reduced phase space of the complex pre-symplectic manifold $(\mathcal{O}\mathfrak{p}_{\Sigma}(\mathrm{G}), \tau_{\mathrm{G}})$ is canonically isomorphic to $(\mathcal{O}\mathfrak{p}_{\Sigma}(\textnormal{PSL}_{2}(\mathbb{C})), \tau_{\textnormal{PSL}_{2}(\mathbb{C})}).$
\end{corollary}

\begin{proof}
This follows from a general fact in pre-symplectic geometry.  Suppose $(M, \tau_{M})$ is a complex pre-symplectic manifold.  Let $(N, \tau_{N})$ be a complex symplectic manifold and
\begin{align}
\iota: (N, \tau_{N}) \rightarrow (M, \tau_{M})
\end{align}
a holomorphic embedding such that $\iota^{\star}\tau_{M}=\tau_{N}.$

Suppose further that the dimension of $N$ is equal to the rank of $\tau_{M}.$  If every leaf of the foliation given by $\textnormal{ker}(\tau_{M})$ intersects $\iota(N)$ in exactly one point, then the reduced phase space of $(M,\tau_{M})$ exists, and there is a canonical pre-symplectomorphism
\begin{align}
(N, \tau_{N})\simeq (M^{\textnormal{red}}, \tau_{M}^{\textnormal{red}}).
\end{align}

Since 
\begin{align}
\iota_{\mathrm{G}}: \mathcal{O}\mathfrak{p}_{\Sigma}(\textnormal{PSL}_{2}(\mathbb{C}))\rightarrow \mathcal{O}\mathfrak{p}_{\Sigma}(\mathrm{G})
\end{align}
satisfies these properties, this completes the proof.
\end{proof}

We close with a discussion of the identifications 
\begin{align}
\mathcal{O}\mathfrak{p}_{\Sigma}(\mathrm{G})\simeq \mathcal{B}_{\Sigma}(\mathrm{G})
\end{align}
from Theorem \ref{h base} from the point of view of pre-symplectic geometry.  

Recall that there is an isomorphism of holomorphic vector bundles 
\begin{align}\label{iso cot}
\mathcal{B}_{\textnormal{PSL}(2,\mathbb{C})}(\Sigma)\simeq T^{\star}\mathcal{T}_{\Sigma}
\end{align}
over $\mathcal{T}_{\Sigma}$
via the identification of co-tangent vectors to $\mathcal{T}_{\Sigma}$ with holomorphic quadratic differentials. 

Being the cotangent bundle of a complex manifold, $T^{\star}\mathcal{T}_{\Sigma}$ has a canonical complex symplectic structure $\omega_{\textnormal{can}}.$  The isomorphism \eqref{iso cot} induces a complex symplectic form $\omega_{\mathcal{B}_{\textnormal{PSL}(2,\mathbb{C})}}$ on $\mathcal{B}_{\textnormal{PSL}(2,\mathbb{C})}(\Sigma).$

Consider the holomorphic map of vector bundles
\begin{align}
R: \mathcal{B}_{\Sigma}(\mathrm{G})&\rightarrow \mathcal{B}_{\textnormal{PSL}(2,\mathbb{C})}(\Sigma) \\
(X,\alpha_{1},...,\alpha_{\ell})&\mapsto (X,\alpha_{1}).
\end{align}
The form $\omega_{\mathcal{B}_{G}}=R^{\star}\omega_{\mathcal{B}_{\textnormal{PSL}(2,\mathbb{C})}}$ is a closed, holomorphic $2$-form on $\mathcal{B}_{G}$ of constant rank $6g-6.$
Note that by the proof of Corollary \ref{reduced}, the couple $(\mathcal{B}_{\textnormal{PSL}(2,\mathbb{C})}(\Sigma), \omega_{\mathcal{B}_{\textnormal{PSL}(2,\mathbb{C})}})$ is canonically isomorphic to the reduced phase space of the complex pre-symplectic manifold $(\mathcal{B}_{\Sigma}(\mathrm{G}), \omega_{\mathcal{B}_{G}}).$

We close with the following theorem.

\begin{theorem}\label{sym id}
Let $s$ be a holomorphic Lagrangian section of 
\begin{align}
\pi: \mathcal{O}\mathfrak{p}_{\Sigma}(\textnormal{PSL}(2,\mathbb{C}))\rightarrow \mathcal{T}_{\Sigma}.
\end{align}
Then the biholomorphism
\begin{align}
\phi_{s}: \mathcal{O}\mathfrak{p}_{\Sigma}(\mathrm{G})\rightarrow \mathcal{B}_{\Sigma}(\mathrm{G})
\end{align}
from Theorem \ref{h base}
satisfies
\begin{align}
\phi_{s}^{\star}\omega_{\mathcal{B}_{G}}=\sqrt{-1}\tau_{\mathrm{G}}.
\end{align}
\end{theorem}

\begin{proof} 
Consider the commutative diagram
\begin{center}
\begin{tikzcd}
\mathcal{O}\mathfrak{p}_{\Sigma}(\mathrm{G}) \arrow{r}{\phi_{s}} \arrow{d}{P}
& \mathcal{B}_{\Sigma}(\mathrm{G})  \arrow{d}{R}\\
\mathcal{O}\mathfrak{p}_{\Sigma}(\textnormal{PSL}(2,\mathbb{C})) \arrow{r}{\theta_{s}}
& \mathcal{B}_{\textnormal{PSL}(2,\mathbb{C})}(\Sigma)
\end{tikzcd}
\end{center}
where 
\begin{align}
P: \mathcal{O}\mathfrak{p}_{\Sigma}(\mathrm{G}) \rightarrow \mathcal{O}\mathfrak{p}_{\Sigma}(\textnormal{PSL}(2,\mathbb{C}))
\end{align}
is the canonical quotient map to the reduced phase space of the holomorphic pre-symplectic manifold $(\mathcal{O}\mathfrak{p}_{\Sigma}(\mathrm{G}),\tau_{\mathrm{G}}).$

Since $s$ is a holomorphic Lagrangian section, a theorem of Kawai \cite{KAW96}, later clarified by  Loustau \cite{LOU15}, implies
\begin{align}
\theta_{s}^{\star}\omega_{\mathcal{B}_{\textnormal{PSL}(2,\mathbb{C})}}=\sqrt{-1}\tau_{\textnormal{PSL}(2,\mathbb{C})}.
\end{align}
Since $P$ is the canonical quotient map to the reduced phase space, 
\begin{align}
P^{\star}\circ \theta_{s}^{\star}\omega_{\mathcal{B}_{\textnormal{PSL}(2,\mathbb{C})}}=P^{\star}\sqrt{-1}\tau_{\textnormal{PSL}(2,\mathbb{C})}=\sqrt{-1}\tau_{\mathrm{G}}.
\end{align}
By commutativity of the above diagram and the definition of $\omega_{\mathcal{B}_{\mathrm{G}}},$ this implies
\begin{align}
\sqrt{-1}\tau_{\mathrm{G}}&=P^{\star}\circ \theta_{s}^{\star}\omega_{\mathcal{B}_{\textnormal{PSL}(2,\mathbb{C})}} \\
&= \phi_{s}^{\star}\circ R^{\star} \omega_{\mathcal{B}_{\textnormal{PSL}(2,\mathbb{C})}} \\
&= \phi_{s}^{\star} \omega_{\mathcal{B}_{\mathrm{G}}}.
\end{align}
This completes the proof.
\end{proof}
\textbf{Remark}:  Every Bers' section (see the remark following Theorem \ref{Hubbard}) of the projection $\pi:  \mathcal{O}\mathfrak{p}_{\Sigma}(\textnormal{PSL}(2,\mathbb{C}))\rightarrow \mathcal{T}_{\Sigma}$ is holomorphic Lagrangian, and therefore we obtain a $\mathcal{T}_{\Sigma}$ worth of holomorphic Lagrangian sections which satisfy the hypotheses of Theorem \ref{sym id}.

\bibliography{Notesbib}{}
\bibliographystyle{alpha}

\end{document}